\documentclass[11pt]{article}
\usepackage{amsmath,amssymb,amsthm}
\usepackage{enumitem}
\usepackage[affil-it]{authblk}
\usepackage{a4wide}
\usepackage{color}
\usepackage{hyperref}
\usepackage[vcentermath]{youngtab}
\newtheorem{thm}[equation]{Theorem}

\newtheorem{prop}[equation]{Proposition}
\newtheorem{lemma}[equation]{Lemma}
\theoremstyle{definition}
\newtheorem{defn}[equation]{Definition}
\newtheorem{remark}[equation]{Remark}
\newtheorem{exam}[equation]{Example}
\numberwithin{equation}{section}

\newcounter{tempcounter}

\newcommand{\FF}{\mathbb{F}}  
\newcommand{\ZZ}{\mathbb{Z}}  
\newcommand{\CC}{\mathbb{C}}  
 
\newcommand{\A}{\mathsf{A}} 
\newcommand{\Z}{\mathsf{Z}} % center
\newcommand\chara{\mathsf{char}}

\newcommand{\pb}[1]{\left\{ #1\right\}}
\newcommand{\seq}[1]{\left( #1\right)}

\newcommand\im{\mathsf{im}}

\newcommand\modd[1]{\, (\mathsf{mod} \, #1)}
\newcommand\degg{\, \mathsf{deg} \,}
\newcommand\spann{\, \mathsf{span}}

\newcommand{\hqfg}{\mathcal{H}_q(f,g)}
\newcommand{\hh}{\mathcal{H}}

\newcommand{\Sf}[1][f]{\mathsf{S}_{#1}}
\newcommand{\Tf}[1][q, g, \lambda]{\mathsf{T}_{#1}}
\newcommand{\Af}[1][\lambda, \mu]{\mathsf{A}_{q, f,g}(#1)}
\newcommand{\Bf}[1][\lambda, \mu]{\mathsf{B}_{q, f,g}(#1)}
\newcommand{\Cf}[1][\alpha]{\mathsf{C}_{q, f,g}(#1)}

\title{Quantum generalized Heisenberg algebras and their representations}

\author{Samuel A.\ Lopes\thanks{Partially supported by CMUP, which is financed by national funds through FCT---Funda\c c\~ao para a Ci\^encia e a Tecnologia, I.P., under the project with reference UIDB/00144/2020.}\ }

\author{Farrokh Razavinia\thanks{Supported by FCT, through the grant PD/BD/142959/2018, under POCH funds, co-financed by the European Social Fund and Portuguese National Funds from MEC.}}
\affil{CMUP, Faculdade de Ci\^encias da
Universidade do Porto, Rua do Campo Alegre 687,
4169-007 Porto (Portugal).}

\date{}

\begin{document}

\maketitle

\begin{abstract}
We introduce and study a new class of algebras, which we name \textit{quantum generalized Heisenberg algebras} and denote by $\hqfg$, related to generalized Heisenberg algebras, but allowing more parameters of freedom, so as to encompass a wider range of applications and include previously studied algebras, such as (generalized) down-up algebras. In particular, our class now includes the enveloping algebra of the $3$-dimensional Heisenberg Lie algebra and its $q$-deformation, neither of which can be realized as a generalized Heisenberg algebra.

This paper focuses mostly on the classification of finite-dimensional irreducible representations of quantum generalized Heisenberg algebras, a study which reveals their rich structure. Although these algebras are not in general noetherian, their representations still retain some Lie-theoretic flavor. We work over a field of arbitrary characteristic, although our results on the representations require that it be algebraically closed.
\newline\newline
\textbf{MSC Numbers (2020)}: Primary 17B81; Secondary 81R10,16G99, 16D60, 16S80, 16T20.
\hfill \newline
\end{abstract}

\section{Introduction}\label{S:intro}

This paper introduces and studies a new class of algebras which we name \textit{quantum generalized Heisenberg algebras} (qGHA, for short), as they can be seen simultaneously as deformations and as generalizations of the generalized Heisenberg algebras appearing in \cite{CRM01} and profusely studied thenceforth in the physics literature (see e.g.~\cite{CRMR13}, \cite{BBH11}, \cite{BCG18} and the references therein). In the mathematics literature, generalized Heisenberg algebras were studied mainly in \cite{LZ15}, \cite{LMZ15} and \cite{sL17}. For an overview of their relevance in mathematical physics see the introductory section in \cite{LZ15}.

Our main motivation for introducing a generalization of this class, besides providing a broader framework for the investigation of the underlying physical systems, comes from the observation in \cite{sL17} that the classes of generalized Heisenberg algebras and of (generalized) down-up algebras intersect (see the seminal paper \cite{BR98} on down-up algebras and also \cite{CS04}), although neither one contains the other. The other interesting feature of our study comes from the fact that quantum generalized Heisenberg algebras are generically non-noetherian and we believe that there are yet not enough studies into the representation theory of non-noetherian algebras which are somehow related to deformations of enveloping algebras of Lie algebras, as is the case with quantum generalized Heisenberg algebras. An additional motivation comes from the fact that both the enveloping algebra of the $3$-dimensional Heisenberg Lie algebra and its quantum deformation are quantum generalized Heisenberg algebras but not generalized Heisenberg algebras, so the name seems better justified now.

Let $\FF$ be an arbitrary field. Given a polynomial $f\in\FF[h]$, the generalized Heisenberg algebra associated to $f$, denoted $\hh(f)$, is the unital associative $\FF$-algebra with generators $x$, $y$ and $h$, with defining relations:
\begin{equation*}
hx=xf(h), \quad yh=f(h)y, \quad yx-xy=f(h)-h.
\end{equation*}
In \cite{LZ15}, working over the complex field $\CC$, the authors determined a basis for $\hh(f)$, computed its center, solved the isomorphism problem for this family of algebras and classifyed all the finite-dimensional irreducible representations of $\hh(f)$, whereas in \cite{sL17} it was shown that $\hh(f)$ is isomorphic to a generalized down-up algebra if and only if $\degg f\leq 1$ and that $\hh(f)$ is noetherian if and only if $\degg f= 1$. All the gradings and the automorphism group of $\hh(f)$ were also computed in \cite{sL17}, in case $\degg f> 1$.

In this paper, working over an arbitrary field $\FF$, we introduce a generalization of this class of algebras by deforming and generalizing the relation $yx-xy=f(h)-h$, turning it into a skew-commutation relation and allowing the skew-commutator to equal a generic polynomial, independent of $f$. This will have as a consequence that the new algebras will in general not be conformal (in the sense of the theory of ambiskew polynomial rings---see Subsection~\ref{SS:basic:ambiskew} for details).

\begin{defn}
Let $\FF$ be an arbitrary field and fix $q\in\FF$ and $f,g\in\FF[h]$. The quantum generalized Heisenberg algebra (qGHA, for short), denoted by $\hqfg$, is the $\FF$-algebra generated by $x$, $y$ and $h$, with defining relations:
\begin{equation}\label{E:intro:def:qgha}
hx=xf(h), \quad yh=f(h)y, \quad yx-qxy=g(h).
\end{equation} 
\end{defn}

\subsection{Examples of quantum generalized Heisenberg algebras}\label{SS:intro:eg}

Clearly, the generalized Heisenberg algebras are precisely the qGHA with $q=1$ and $g=f(h)-h$, i.e.\ $\hh(f)=\hh_1(f, f-h)$. Let us consider more general examples.

Recall that, for parameters $\alpha, \beta, \gamma\in\FF$, the down-up algebra $A(\alpha, \beta, \gamma)$ was defined by Benkart and Roby in \cite{BR98} as the unital associative algebra with generators $d$ and $u$ and defining relations:
\begin{equation*}
d^2 u = \alpha dud + \beta ud^2 + \gamma d\quad \mbox{and}\quad
du^2 = \alpha udu + \beta u^2 d + \gamma u. 
\end{equation*}
The motivation for the above relations came from the combinatorics of differential posets but it was realized that down-up algebras included many ubiquitous algebras appearing in Lie theory, including the universal enveloping algebra of the Lie algebra $\mathfrak{sl}_2$ (the down-up algebra $A(2,-1,-2)$) and the universal enveloping algebra of the 3-dimensional Heisenberg Lie algebra $\mathfrak h$ (the down-up algebra $A(2,-1,0)$), as well as many deformations (including Witten's 7-parameter deformations of the universal enveloping algebra of $\mathfrak{sl}_2$, see \cite{gB99}).

In \cite{CS04}, Cassidy and Shelton generalized this construction and introduced the generalized down-up algebra $L(v,r,s,\gamma)$ as the unital associative algebra generated by $d$, $u$ and $h$ with defining relations
\begin{equation*}
dh-rhd +\gamma d =0, \quad \quad hu-ruh +\gamma u=0 \quad \mbox{and}\quad du-sud +v(h)=0,
\end{equation*}
where $r, s, \gamma\in\FF$ and $v\in\FF[h]$.
Generalized down-up algebras include all down-up algebras, as long as the polynomial $h^2-\alpha h-\beta$ has roots in $\FF$. Indeed, if $\alpha=r+s$ and $\beta=-rs$, then it is easy to see that $A(\alpha, \beta, \gamma)\simeq L(-h,r,s,-\gamma)$. Conversely, any generalized down-up algebra $L(v,r,s,\gamma)$ with $\degg v=1$ is a down-up algebra. Moreover, the following are generalized down-up algebras: the algebras \textit{similar to the enveloping algebra of $\mathfrak{sl}_2$} defined by Smith \cite{spS90}, Le Bruyn conformal $\mathfrak{sl}_2$ enveloping algebras \cite{lLB95}, and Rueda's algebras \textit{similar to the enveloping algebra of $\mathfrak{sl}_2$} \cite{sR02}.

We have the following important observation, which implies that the class of generalized down-up algebras coincides with the class of quantum generalized Heisenberg algebras $\hqfg$ such that $\degg f\leq 1$.

\begin{prop}\label{P:intro:gduaRqgha}
Let $r, s, \gamma\in\FF$ and $v\in\FF[h]$. Then the generalized down-up algebra $L(v,r,s,\gamma)$ is isomorphic to the quantum generalized Heisenberg algebra $\hh_s(rh-\gamma, -v)$. In particular, the down-up algebra $A(\alpha, \beta, \gamma)$ is isomorphic to the quantum generalized Heisenberg algebra $\hh_s(rh+\gamma, h)$, where $\alpha=r+s$ and $\beta=-rs$.

Conversely, any quantum generalized Heisenberg algebra $\hqfg$ such that $f(h)=ah+b$, with $a, b\in\FF$, is a generalized down-up algebra of the form $L(-g,a,q,-b)$.
\end{prop}

In view of this result and the fact that the generalized down-up algebras have been extensively studied, we will often focus on quantum generalized Heisenberg algebras $\hqfg$ with $\degg f>1$.
In fact, in~\cite{LR20pr} we shall see that if a quantum generalized Heisenberg algebra $\hqfg$ is isomorphic to a generalized down-up algebra, then necessarily $\degg f\leq1$.

\subsection{Isomorphisms among quantum generalized Heisenberg algebras}\label{SS:intro:isos}

The isomorphism problem for quantum generalized Heisenberg algebras requires technical results concerning the noetherianity of the qGHA and will be tackled in \cite{LR20pr}. It will be proved there that the isomorphism relation can be phrased in very concrete geometric terms, very much like in \cite{BJ01}. It will follow in particular that, in case $q\neq 0$ and $\degg f>1$, the parameter $q$, as well as the integers $\degg f$ and $\degg g$, are isomorphism invariant, showing that qGHA are indeed a vast generalization of generalized Heisenberg algebras and generalized down-up algebras.
However, our results in the present paper on the center and the description of the finite-dimensional simple modules are enough to show that qGHA are a much broader class which properly contains all generalized Heisenberg algebras and down-up algebras, as their finite-dimensional irreducible representations depend heavily on $q$, $f$ and $g$.

\subsection{Organization of the paper}\label{SS:intro:org}

In Section~\ref{S:basic} we relate quantum generalized Heisenberg algebras to known constructions, such as Ore extensions, ambiskew polynomial rings and generalized Weyl algebras. From these we deduce the basic properties of the algebras $\hqfg$, including a PBW-type basis, necessary and sufficient conditions for  $\hqfg$ to be a domain, as well as a description of the center, under a few restraints.

The main results of this paper are in Section~\ref{S:irreps}, where we classify, up to isomorphism, all finite-dimensional simple representations of $\hqfg$, assuming only that $q\neq 0$ and that the base field is algebraically closed, although of arbitrary characteristic. We also explicitly describe all possible isomorphisms between these modules.

We conclude the paper with a series of questions for further study, in Section~\ref{S:concluding}.

\subsection{Conventions and notation}\label{SS:intro:notation}

Throughout the paper, unless otherwise noted, $\FF$ will denote an arbitrary field, with multiplicative group denoted by $\FF^*$ and algebraic closure $\overline\FF$. The integers, nonnegative integers and positive integers will be denoted by $\ZZ$, $\ZZ_{\geq 0}$ and $\ZZ_{>0}$, respectively. Given a set $E$, the identity map on $E$ will be denoted by $1_E$. If $A$ is a ring, the center of $A$ is denoted $\Z(A)$ and $\mathsf{C}_A(a)$ denotes the centralizer of $a$ in $A$. An element $\theta\in A$ is \textit{normal} if $\theta A=A\theta$.

The relation $hx=xf(h)$ implies that $hx^2=x^2 f(f(h))$ and similarly for higher powers of $x$. To avoid confusion with differential or multiplicative notation, we denote compositional powers of a polynomial as follows:
\begin{equation}
\underbrace{p(p(\cdots (p}_{\text{$k$ times}}(h))))=p^{[k]}(h),
\end{equation}
where $p\in\FF[h]$. Let $\sigma$ be the unital algebra endomorphism of $\FF[h]$ which maps $h$ to $f(h)$. Then $f^{[k]}(h)=\sigma^k(h)$ and we also use the latter notation. Thus, for example, $hx^k=x^k \sigma^k(h)=x^k f^{[k]}(h)$, for all $k\geq 0$.

Whenever possible, we will exploit the symmetry between $x$ and $y$ in the defining relations \eqref{E:intro:def:qgha}. This can be made precise by noting that there is an anti-automorphism of order 2, $\iota:\hqfg\longrightarrow\hqfg$, fixing $h$ and interchanging $x$ and $y$. Applying $\iota$ to an equation in $\hqfg$ will reverse the roles of $x$ and $y$ at the cost of inverting the order of multiplication.

\bigskip

\noindent  {\bf Acknowledgments:}  \  We thank David Jordan for pointing out the contents of Remark~\ref{R:dJ}.

\section{The structure of quantum generalized Heisenberg algebras}\label{S:basic}

in this section we relate qGHA with ambiskew polynomial rings and weak generalized Weyl algebras, describe bases, gradings and the center, and determine when $\hqfg$ is a domain.

\subsection{Quantum generalized Heisenberg algebras and ambiskew polynomial rings}\label{SS:basic:ambiskew}

As seen above, quantum generalized Heisenberg algebras define a vast class of algebras which includes as a special case the (generalized) down-up algebras. The natural setting for constructing and studying qGHA is that of ambiskew polynomial rings, as defined by Jordan over a series of papers (see \cite{dJ00}, \cite{JW13}, \cite{dJ19} and the references therein). Indeed, in \cite{dJ00} Jordan adjusts the original definition of an ambiskew polynomial ring precisely to include the non-noetherian down-up algebras. 

The construction of an ambiskew polynomial ring involves a 2-step Ore extension of non-automorphism type, where on the first step the coefficients are added on the left and then, on the second step, on the right (the reverse construction is possible as well). Since the theory of Ore extension is not essential in this paper, we omit their definition and well-known properties, referring the interested reader to \cite{GW89}.

The definition of an ambiskew polynomial ring which suits our needs is \cite[Sec.\ 7]{dJ00} (see also \cite{dJ19}) and it follows immediately from this definition that $\hqfg$ is precisely an ambiskew polynomial ring over $\FF[h]$, associated with the algebra endomorphism $\sigma: h\mapsto f(h)$. Thus we get the following, which generalizes \cite[Lemma 1]{LZ15}.

\begin{lemma}\label{L:basic:basis}
The quantum generalized Heisenberg algebra $\hqfg$ is an ambiskew polynomial ring of endomorphism type over $\FF[h]$. For any basis $\pb{v_j}_{j\in \ZZ_{\geq 0}}$ of $\FF[h]$, the set $\pb{x^i v_j y^k \mid i, j, k\in \ZZ_{\geq 0}}$ is a basis of $\hqfg$. In particular, for each $j\geq 0$, we can take $v_j$ to be any polynomial in $h$ of degree $j$.
\end{lemma}

It should be remarked that Lemma~\ref{L:basic:basis} implies that $x$ is not a left zero divisor and that $y$ is not a right zero divisor. Moreover, viewing $\hqfg$ as a 2-step Ore extension (or as an ambiskew polynomial ring) also helps to determine when $\hqfg$ is a domain.

\begin{prop}\label{P:basic:domain}
The qGHA $\hqfg$ is a domain if and only if $q\neq 0$ and  $\degg f\geq 1$.
\end{prop}
\begin{proof}
Suppose first that $q\neq 0$ and $\degg f\geq 1$. Then it follows from the theory of Ore extensions that $\hqfg$ is a domain (compare~\cite[Sec.\ 4, par.\ 1]{dJ19}). Next, we show that these conditions are necessary for $\hqfg$ to be a domain.

If $f\in\FF$, then $h-f\neq 0$ and $(h-f)x=xf-fx=0$, so $\hqfg$ is not a domain in this case. Now assume that $q=0$ and $f\not\in\FF$. Then we have the relation $yx=g(h)$ so if $g=0$ we again deduce that $\hqfg$ is not a domain. Suppose thus that $g\neq 0$. 

Let $\pi\circ\sigma:\FF[h]\longrightarrow\FF[h]/(g(h))$ be the composition of $\sigma$ with the canonical surjection $\pi:\FF[h]\longrightarrow\FF[h]/(g(h))$, where $(g(h))=g(h)\FF[h]$ is the ideal generated by $g$. Since $g\neq 0$, the quotient $\FF[h]/(g(h))$ is finite dimensional. On the other hand, since $\degg f\geq 1$, the map $\sigma$ is injective and in particular its image $\im\, \sigma$ is infinite dimensional. This implies that $\im\, \sigma\cap (g(h))\neq \{0\}$. Hence, there exist nonzero polynomials $P_0, P_1\in\FF[h]$ with $gP_1=\sigma(P_0)$. Thus,
\begin{equation*}
y(xP_1y-P_0)=yxP_1y-yP_0=(gP_1-\sigma(P_0))y=0.
\end{equation*}
By Lemma~\ref{L:basic:basis}, both $y$ and $xP_1y-P_0$ are nonzero, so $\hqfg$ is not a domain.  
\end{proof}

\begin{remark}\label{R:dJ}
In \cite{dJ19} Jordan studies simplicity criteria for algebras of this type although, generally speaking, quantum generalized Heisenberg algebras are not simple (compare \cite[Lem.\ 2.9, Thm.\ 3.8]{dJ19}). However, had we instead been working over the rational function field $\FF(h)$ with $\FF$ of characteristic $0$, $q\neq 0$, $f=h^a$, $g=h^b$ and $a\geq b$, then we would obtain a simple algebra. 
\end{remark}

\subsection{Quantum generalized Heisenberg algebras and (weak) generalized Weyl algebras}\label{SS:basic:wGWA}

Another prominent class of algebras arising in connection with Lie theory and quantum algebras is that of generalized Weyl algebras. These were defined by Bavula in \cite{vb91, vb92} and subsequently studied extensively by himself and many other authors. The references are too many to mention but we single out the work \cite{LMZ15} as it most directly pertains to our setting. 

In its simplest form, a generalized Weyl algebra is an algebra attached to the following data: a unital associative algebra $A$, an algebra automorphism $\psi$ of $A$ and a central element $t\in\Z(A)$. The generalized Weyl algebra $A(\psi, t)$ is the associative algebra generated over $A$ by elements $x, y$, subject to the relations:
\begin{equation}\label{E:basic:def:gwa}
ax=x\psi(a), \quad\quad ya=\psi(a)y, \quad\quad xy=t \quad \mbox{and}\quad yx=\psi(t), \quad\quad\mbox{for all $a\in A$.}
\end{equation}
In \cite{LMZ15} the term \textit{weak} generalized Weyl algebra was coined to refer to an algebra $A(\psi, t)$ as above, but such that $\psi$ is only assumed to be a (unital) algebra endomorphism of $A$. Then one can see that $\hqfg$ is a weak generalized Weyl algebra $A(\tilde\sigma, t)$ over the polynomial ring $A=\FF[h, t]$, taking $\tilde\sigma$ to be the endomorphism defined by $\tilde\sigma(h)=f(h)$ and $\tilde\sigma(t)=qt +g(h)$. Indeed, the relation $xy=t$ from \eqref{E:basic:def:gwa} makes the variable $t$ from the polynomial ring $A=\FF[h, t]$ redundant and then the relation $yx=\tilde\sigma(t)=qt+g(h)$ becomes $yx=qxy+g(h)$, the third relation in the definition of $\hqfg$ in \eqref{E:intro:def:qgha}.

Generalized Weyl algebras carry a natural $\ZZ$-grading obtained by setting $x$ in degree $1$, the elements of $A$ in degree $0$ and $y$ in degree $-1$. We use the term \textit{weight} to refer to the degree of a homogeneous element with respect to this grading. It is clear that this grading carries over to weak generalized Weyl algebras and in particular we see that $\hqfg$ inherits this $\ZZ$-grading:
\begin{equation}\label{E:basic:weight_space_dec}
\hqfg=\bigoplus_{k\in\ZZ} \hqfg_k, 
\end{equation}
where, for $k\geq 0$,
\begin{equation*}
 \hqfg_0=\bigoplus_{i\geq 0}x^i \FF[h]y^i, \quad  \hqfg_k=x^k\, \hqfg_0 \quad \mbox{and}\quad \hqfg_{-k}=\hqfg_0\, y^{k}.
\end{equation*}

\subsection{Conformality and the center of $\hqfg$}\label{SS:basic:center}

In \cite{JW96}, Jordan and Wells introduced the notion of conformal ambiskew polynomial rings, although some care is needed because this is not an isomorphism-invariant property, as remarked in \cite{CL09}. Viewing quantum generalized Heisenberg algebras as ambiskew polynomial rings, as above, we say that $\hqfg$ is \textit{conformal} if there is some $a\in \FF[h]$ such that $g(h)=\sigma(a)-qa$. 

Conformality has important implications. Under this assumption, define $Z=q(xy-a)=yx-\sigma(a)$. Then the following easy computations show that $Z$ is normal:
\begin{equation}\label{E:basic:center:Zcomm}
hZ=Zh, \quad Zx=qxZ, \quad \mbox{and}\quad yZ=qZy. 
\end{equation}
If in addition $q=1$, then the above shows that $Z$ is central. For example, the generalized Heisenberg algebras $\mathcal{H}(f)$ are always conformal, as by definition these are of the form $\mathcal{H}_1(f, f(h)-h)$ so we can take $a=h$ and thence $Z$ is central in $\mathcal{H}(f)$. In contrast, we will see shortly that quantum generalized Heisenberg algebras can have a trivial center.

\begin{lemma}\label{L:basic:conformal}
Let $\hqfg$ be a quantum generalized Heisenberg algebra and $k\geq 0$. Set $\theta_k=\sum_{i=0}^{k-1}q^i \sigma^{k-1-i}(g)$, where by convention $\theta_0=0$. Then the following hold:
\begin{enumerate}[label={(\alph*)}]
\item $yx^k=q^kx^ky+x^{k-1}\theta_k$ and $y^k x=q^kx y^k+\theta_k y^{k-1}$.\label{L:basic:conformal:a}
\item For any $p(h)\in\FF[h]$, $(x^k p(h)y^k)x=x(q^k x^k\sigma(p(h))y^k + x^{k-1}p(h)\theta_k y^{k-1})$ and $y(x^k p(h)y^k)=(q^k x^k\sigma(p(h))y^k + x^{k-1}p(h)\theta_k y^{k-1})y$.\label{L:basic:conformal:b}
\item If $g(h)=\sigma(a)-qa$, for some $a\in \FF[h]$, then $\theta_k=\sigma^k(a)-q^ka$.\label{L:basic:conformal:c}
\end{enumerate}
\end{lemma}
\begin{proof}
Parts \ref{L:basic:conformal:a} and \ref{L:basic:conformal:c} appear in \cite[Sec.\ 2.4]{dJ00} and can readily be proved by induction on $k$; part \ref{L:basic:conformal:b} follows immediately from \ref{L:basic:conformal:a}.
\end{proof}

For the remainder of this section we will assume that $\degg f>1$. Otherwise, as we have previously remarked, we obtain a generalized down-up algebra, whose center has been studied elsewhere. This assumption avoids several technical difficulties which have been dealt with already in the literature on 
generalized down-up algebras.

\begin{lemma}\label{L:basic:centralizer:h}
Let $\hqfg$ be a quantum generalized Heisenberg algebra with $\degg f>1$. Then the centralizer of $h$ is $\hqfg_0$.
\end{lemma}
\begin{proof}
First, it is clear that $\hqfg_0\subseteq\mathsf{C}_{\hqfg}(h)$, as $h(x^k p(h)y^k)=x^k \sigma^k(h)p(h)y^k=(x^k p(h)y^k)h$, for all $k\geq 0$ and all $p(h)\in\FF[h]$. It remains to prove the converse inclusion.

Since $h$ is homogenous of weight $0$, its centralizer algebra is necessarily the sum of its weight components. Thus, to finish the proof, it is enough to argue that $h(x^k p(h)y^\ell)=(x^k p(h)y^\ell)h$ for some nonzero $p(h)$ implies that $k=\ell$. So assume that $p(h)\neq 0$ and $h(x^k p(h)y^\ell)=(x^k p(h)y^\ell)h$. Then, using the relations \eqref{E:intro:def:qgha} and Lemma~\ref{L:basic:basis}, we deduce that $\sigma^k(h)p(h)=\sigma^\ell(h)p(h)$. As $\FF[h]$ is a domain, $p(h)\neq 0$ and $\sigma$ is injective (because $\degg f\geq 1$), we deduce that there is $j=|k-\ell|\geq 0$ such that $\sigma^j(h)=h$. If we had $j>0$, the latter would imply that $\degg f= 1$, so indeed $j=0$ and $k=\ell$.
\end{proof}

The following result generalizes \cite[Thm.\ 4]{LZ15}, in case $\degg f>1$.

\begin{prop}\label{P:basic:center:conf}
Let $\hqfg$ be a quantum generalized Heisenberg algebra with $\degg f>1$.
%requires \usepackage{enumitem}
\begin{enumerate}[label={(\alph*)}]
\item If $q$ is not a root of unity, then $\Z(\hqfg)=\FF$.
\item If $q$ is a primitive $\ell$-th root of unity and $\hqfg$ is conformal, with $g(h)=\sigma(a)-qa$, then $\Z(\hqfg)=\FF[Z^\ell]$, where $Z=q(xy-a)$.
\end{enumerate} 
\end{prop}
\begin{proof}
Since $\Z(\hqfg)\subseteq \mathsf{C}_{\hqfg}(h)$, Lemma~\ref{L:basic:centralizer:h} shows that $\Z(\hqfg)\subseteq \hqfg_0$.
Thus, assume that $z=\sum_{i=0}^k x^i p_i(h)y^i$ is central, with $p_i(h)\in\FF[h]$ for all $i$ and $p_k(h)\neq 0$. Using Lemma~\ref{L:basic:conformal}\ref{L:basic:conformal:b} and comparing the decompositions of $xz$ and $zx$ in the basis $\pb{x^i h^j y^k}_{i, j, k}$, we deduce that $q^k\sigma(p_k(h))=p_k(h)$. This immediately implies that $p_k(h)\in\FF^*$ and $q^k=1$.
In particular, if $q$ is not a root of unity, then the above forces $k=0$ and $z\in\FF^*$, so in this case $\Z(\hqfg)=\FF$.

Now, assume that $q$ is a primitive $\ell$-th root of unity, for some $\ell>0$ and that $\hqfg$ is conformal, with $g(h)=\sigma(a)-qa$ and $Z=q(xy-a)$. Then the relations \eqref{E:basic:center:Zcomm} imply that $Z^\ell$ is central.

Let us introduce the notation $\hqfg_{0, j}=\bigoplus_{i=0}^{j} x^i \FF[h]y^i$, for $j\geq 0$. By Lemma~\ref{L:basic:conformal}\ref{L:basic:conformal:a}, these subspaces form an increasing filtration of $\hqfg_{0}$. Given $\alpha, \beta\in \hqfg_{0}$, we will write $\alpha=\beta+\mathrm{LOT}(j)$ if $\alpha-\beta\in \hqfg_{0, j}$. Then $(xy)^j=q^{\lambda_j}x^jy^j +\mathrm{LOT}(j-1)$, where $\lambda_j=\frac{j(j-1)}{2}$. Since $xy$ and $a\in\FF[h]$ commute, the binomial theorem implies that $Z^j=q^{j+\lambda_j}x^jy^j + \mathrm{LOT}(j-1)$, for all $j\geq 0$. As $q\neq 0$, it follows that the powers of $Z$ are linearly independent and $\FF[Z^\ell]$ is indeed a polynomial algebra. Hence, we have established the inclusion $\Z(\hqfg)\supseteq \FF[Z^\ell]$.

Let $z=\sum_{i=0}^k x^i p_i(h)y^i$ be as above and assume that $z$ commutes with $x$. We will show by induction on $k$ that $z\in\FF[Z^\ell]$. If $k=0$, then we had already observed that $zx=xz$ implies that $z\in\FF^*\subseteq \FF[Z^\ell]$. Suppose now that $k\geq 1$ and that the claim holds for smaller values of $k$. Since $zx=xz$, we can conclude, as above, that $q^k=1$ and $p_k(h)=\lambda\in\FF^*$. In particular, $z=\lambda x^ky^k+\mathrm{LOT}(k-1)$. Therefore, there is $\mu\in\FF^*$ such that $z-\mu Z^k\in \hqfg_{0, k-1}$. As $q^k=1$, $\ell$ divides $k$ and $Z^k=\left(Z^\ell\right)^{k/\ell}\in \FF[Z^\ell]\subseteq \Z(\hqfg)$. Hence, $z-\mu Z^k$ commutes with $x$ and the induction hypothesis applies, so that $z-\mu Z^k\in\FF[Z^\ell]$. Thence, $z\in\FF[Z^\ell]$ and the induction is complete.
\end{proof}

\section{The finite-dimensional simple $\hqfg$-modules}\label{S:irreps}

In this section we completely classify the finite-dimensional simple $\hqfg$-modules for all polynomials $f, g\in\FF[h]$, assuming only that $q\neq 0$ and $\FF$ is algebraically closed, which we abbreviate by writing $\FF=\overline\FF$. In particular, this study generalizes and unifies the classification of finite-dimensional simple modules over down-up algebras, generalized down-up algebras and generalized Heisenberg algebras, which has been carried out over the series of references \cite{BR98}, \cite{CM00}, \cite{CS04}, \cite{LZ15}, \cite{spS90}, \cite{lLB95} and \cite{sR02}, to name just a few. 

Our main results are Theorem~\ref{T:irreps:fdiml:all}, which classifies finite-dimensional simple modules over qGHA, and the results in Subsection~\ref{SS:irreps:isos} describing the isomorphisms between those. The methods we use are akin to the ones in \cite{LZ15}, although we work over a field of arbitrary characteristic and now the dynamics of the action also involves both the deformation parameter $q$ and the polynomial $g$. One crucial difference between our approach and that in \cite{LZ15} is that, in the latter, the existence of a non-trivial central element (whose existence is guaranteed by the conformality of generalized Heisenberg algebras, but which does not exist in general for qGHA, see Subsection~\ref{SS:basic:center}) is used in a strong way to capture results about the simple representations, by way of Schur's lemma. In the more general setting of qGHA, the center may be trivial (see Proposition~\ref{P:basic:center:conf}), so we need other methods to handle that shortfall.

Throughout this section we suppose that $\FF=\overline\FF$ and fix the parameter $q\neq 0$ and the polynomials $f, g\in\FF[h]$.

\subsection{Doubly-infinite weight $\hqfg$-modules}\label{SS:irreps:Verma}

As in \cite{LZ15}, we set 
\begin{equation*}
\Sf=\pb{\lambda:\ZZ\longrightarrow\FF\mid f(\lambda(i))=\lambda(i+1), \ \mbox{for all $i\in\ZZ$}}
\end{equation*}
and, for $\lambda\in\Sf$, let $|\lambda|\geq 0$ be the generator of the additive subgroup 
\begin{equation*}
\pb{k\in\ZZ\mid \lambda(i)=\lambda(i+k), \ \mbox{for all $i\in\ZZ$}} \subseteq\ZZ.
\end{equation*}
By \cite[Lem.\ 6]{LZ15}, if $|\lambda|\neq 0$, then $\lambda(i)=\lambda(j)\iff$ $|\lambda|$ divides $i-j$. 

Given $\lambda\in\Sf$, we define
\begin{equation*}
\Tf=\pb{\mu:\ZZ\longrightarrow\FF\mid \mu(i+1)=q\mu(i)+g(\lambda(i)), \ \mbox{for all $i\in\ZZ$}}.
\end{equation*}

The next result shows in particular that the set $\Tf$ can be parametrized by $\FF$.

\begin{lemma}\label{L:irreps:Verma:mu}
Suppose $q\neq 0$. Let $\lambda\in \Sf$, $k\in\ZZ$ and $\alpha\in\FF$. 
\begin{enumerate}[label=(\alph*)]
\item Define $\mu:\ZZ\longrightarrow\FF$ by
\begin{equation}\label{E:irreps:Verma:mu}
\mu(i+k)=\begin{cases}
q^i\alpha+\sum\limits_{j=0}^{i-1}q^j g(\lambda(k+i-j-1)) & \mbox{if $i\geq 0$;}\\[10pt]
q^i\alpha-\sum\limits_{j=i}^{-1}q^jg(\lambda(k+i-j-1)) & \mbox{if $i\leq 0$.} 
\end{cases} 
\end{equation}
Then $\mu\in\Tf$ and $\mu(k)=\alpha$.\label{L:irreps:Verma:mu:a}
\item Conversely, suppose $\mu\in\Tf$ and $\mu(k)=\alpha$. Then $\mu$ is given by \eqref{E:irreps:Verma:mu} above.\label{L:irreps:Verma:mu:b}
\end{enumerate}
\end{lemma}
\begin{proof}
Part \ref{L:irreps:Verma:mu:a} can be verified directly by separately considering the cases $i\geq 0$ and $i\leq -1$. For \ref{L:irreps:Verma:mu:b}, suppose $\mu\in\Tf$ with $\mu(k)=\alpha$. We will prove that \eqref{E:irreps:Verma:mu} holds by induction on $|i|$. Since the case $i=0$ is tautological, we proceed by induction.

Let $i\geq 0$ and suppose that \eqref{E:irreps:Verma:mu} holds for $i$. We will prove that it also holds for $i+1$. We have
\begin{align*}
\mu(i+1+k)&=q\mu(i+k)+g(\lambda(i+k))\\
&=q^{i+1}\alpha+\sum\limits_{j=0}^{i-1}q^{j+1}g(\lambda(k+i-j-1))+g(\lambda(i+k))\\
&= q^{i+1}\alpha+\sum\limits_{j=1}^{i}q^{j}g(\lambda(k+i-j))+g(\lambda(i+k))\\
&=q^{i+1}\alpha+\sum\limits_{j=0}^{i}q^{j}g(\lambda(k+i-j)).
\end{align*}

Now assume that \eqref{E:irreps:Verma:mu} holds for $i$, with $i\leq0$. We will show that it still holds for $i-1$. Since $\mu(i+k)=q\mu(i+k-1)+g(\lambda(i+k-1))$, we have 
\begin{align*}
\mu(i-1+k)&=q^{-1}\mu(i+k)-q^{-1}g(\lambda(i-1+k))\\
&=q^{i-1}\alpha-\sum\limits_{j=i}^{-1}q^{j-1}g(\lambda(k+i-j-1))-q^{-1}g(\lambda(i-1+k))\\
&=q^{i-1}\alpha-\sum\limits_{j=i-1}^{-2}q^{j}g(\lambda(k+i-j-2))-q^{-1}g(\lambda(i-1+k))\\
&=q^{i-1}\alpha-\sum\limits_{j=i-1}^{-1}q^{j}g(\lambda(k+i-j-2)).
\end{align*}
\end{proof}

\begin{lemma}\label{L:irreps:Verma:periodicl}
Suppose $q\neq 0$. Let $\lambda\in\Sf$ and $\mu\in\Tf$.
\begin{enumerate}[label=(\alph*)]
\item For all $k,i\in\ZZ$, $\mu(i+k|\lambda|)-\mu(i)=q^i(\mu(k|\lambda|)-\mu(0))$.  \label{L:irreps:Verma:periodicl:a}
\item The set $\pb{k\in\ZZ\mid\mu(k|\lambda|)=\mu(0)}$ is an additive subgroup of $\ZZ$.\label{L:irreps:Verma:periodicl:b}
\item \label{L:irreps:Verma:periodicl:c} Let $m\in\ZZ$ be a generator of the subgroup in \ref{L:irreps:Verma:periodicl:b} above.Then, for any $j, k, k'\in\ZZ$,
\begin{equation*}
\mu(j+k|\lambda|)=\mu(j+k'|\lambda|) \quad \mbox{if and only if \quad $m\, |\, k-k'.$} 
\end{equation*}%
\end{enumerate}
\end{lemma}
\begin{proof}
The proof of \ref{L:irreps:Verma:periodicl:a} is by induction on $|i|$. In case $i=0$ there is nothing to prove, so suppose $i\geq 0$ and \ref{L:irreps:Verma:periodicl:a} holds for $i$. Then
\begin{align*}
\mu(i+1+k|\lambda|)-\mu(i+1)&=q\mu(i+k|\lambda|)+g(\lambda(i+k|\lambda|))-q\mu(i)-g(\lambda(i))\\
&=q(\mu(i+k|\lambda|)-\mu(i))\\
&=q^{i+1}(\mu(k|\lambda|)-\mu(0)).
\end{align*}
Now suppose $i\leq 0$. We have
\begin{align*}
\mu(i-1+k|\lambda|)-\mu(i-1)&=q^{-1}\seq{\mu(i+k|\lambda|)-g(\lambda(i-1+k|\lambda|))-\mu(i)
+g(\lambda(i-1))}\\
&=q^{-1}(\mu(i+k|\lambda|)-\mu(i))\\
&=q^{i-1}(\mu(k|\lambda|)-\mu(0)).
\end{align*}

The set in \ref{L:irreps:Verma:periodicl:b} contains $0$. Suppose $a, b\in\pb{k\in\ZZ\mid\mu(k|\lambda|)=\mu(0)}$. Then, using \ref{L:irreps:Verma:periodicl:a}, we get
\begin{align*}
\mu(0)&=\mu(a|\lambda|)=\mu((a-b)|\lambda|+b|\lambda|)=
\mu((a-b)|\lambda|)+q^{(a-b)|\lambda|}(\mu(b|\lambda|)-\mu(0))\\
&=\mu((a-b)|\lambda|).
\end{align*}
So $a-b\in\pb{k\in\ZZ\mid\mu(k|\lambda|)=\mu(0)}$.

Finally, for \ref{L:irreps:Verma:periodicl:c}, we again use \ref{L:irreps:Verma:periodicl:a}:
\begin{align*}
\mu(j+k|\lambda|)=\mu(j+k'|\lambda|+(k-k')|\lambda|)= \mu(j+k'|\lambda|)+q^{j+k'|\lambda|}(\mu((k-k')|\lambda|)-\mu(0)).
\end{align*}
Thus, 
\begin{equation*}
\mu(j+k|\lambda|)=\mu(j+k'|\lambda|)\iff \mu((k-k')|\lambda|)=\mu(0)\iff  k-k'\in m\ZZ. 
\end{equation*}
\end{proof}

In view of Lemma~\ref{L:irreps:Verma:periodicl}\ref{L:irreps:Verma:periodicl:b}, given $\mu\in\Tf$, with $\lambda\in\Sf$ and $q\neq 0$, we define $|\mu|\geq 0$ so that
\begin{equation*}
|\mu|\,\ZZ= \pb{k\in\ZZ\mid\mu(k|\lambda|)=\mu(0)}.
\end{equation*}

Recall that $\sigma$ is the algebra endomorphism of $\FF[h]$ with $\sigma(h)=f$. Thus, $\sigma^k(t(h))=t(\sigma^k(h))=t\seq{f^{[k]}(h)}$, for all $k\geq 0$ and $t\in\FF[h]$.

\begin{lemma}\label{L:weightmods}
Let $V$ be an $\hqfg$-module and suppose that there are $v\in V$ and $\alpha,\beta \in \FF$ such that $h v=\alpha v$ and $(xy) v=\beta v$. Then, for all $k, \ell\geq 0$, we have
\begin{enumerate}[label=(\alph*)]
\item $h (x^kv)=f^{[k]}(\alpha)x^kv$; \label{L:weightmods:a}
\item $(xy) (x^kv)=\seq{q^k\beta+\sum\limits_{i=0}^{k-1}q^i g\seq{f^{[k-i-1]}(\alpha)}}x^kv$;\label{L:weightmods:b}
\item $(xy)(x^{k\ell}v)=\seq{q^{k\ell}\beta+\sum\limits_{i=0}^{\ell-1}\sum\limits_{j=0}^{k-1}q^{i+j\ell}g\seq{f^{[(k-j)\ell-1-i]}(\alpha)}}x^{k\ell}v$.\label{L:weightmods:c}
\end{enumerate}
\end{lemma}
\begin{proof}
Part \ref{L:weightmods:a} is implied by the relation $hx^k=x^k f^{[k]}(h)$ and \ref{L:weightmods:b} follows similarly from Lemma~\ref{L:basic:conformal}\ref{L:basic:conformal:a}. Finally, part \ref{L:weightmods:c} is a direct consequence of \ref{L:weightmods:b} and the fact that any $0\leq n\leq k\ell-1$ can be written uniquely in the form $n=i+j\ell$ for $0\leq i\leq \ell-1$ and $0\leq j\leq k-1$.
\end{proof}

Drawing motivation from Lie theory, an element $v$ as in Lemma~\ref{L:weightmods} should be thought of as a \textit{weight vector}, relative to the commuting elements $h$ and $xy$. An $\hqfg$-module having a basis consisting of weight vectors should be thought of as a \textit{weight module}. We will construct universal weight modules next.

Given $\lambda\in\Sf$ and $\mu\in\Tf$ we define the $\hqfg$-module $\Af$ by setting $\Af=\FF[t^{\pm 1}]$, as vector spaces, with action given by
\begin{equation}\label{E:reps:def:Af}
ht^i=\lambda(i)t^i, \quad xt^i=t^{i+1}, \quad yt^i=\mu(i)t^{i-1}, \quad \mbox{for all $i\in\ZZ$.}
\end{equation}
Let us see that \eqref{E:reps:def:Af} does indeed define an $\hqfg$-module. For $i\in\ZZ$,
\begin{align*}
(hx-xf(h))t^i=ht^{i+1}-xf(\lambda(i))t^i=(\lambda(i+1)-f(\lambda(i)))t^{i+1}=0,
\end{align*}
by the definition of $\Sf$. Similarly, $(yh-f(h)y)t^i=0$. Finally,	
\begin{align*}
(yx-qxy-g(h))t^i=yt^{i+1}-q\mu(i)xt^{i-1}-g(\lambda(i))t^i=(\mu(i+1)-q\mu(i)-g(\lambda(i)))t^i=0,
\end{align*}
by the definition of $\Tf$.

Recall that $\hqfg$ has an order 2 anti-automorphism $\iota$ which fixes $h$ and interchanges $x$ and $y$. We can use $\iota$ to define a module structure on the dual of an $\hqfg$-module. Concretely, let $V$ be an $\hqfg$-module. Then we can define on the dual space $V^*$ the action
\begin{equation*}
(a\cdot f)(v)=f(\iota(a)\cdot v), \quad\mbox{for all $a\in\hqfg$, $f\in V^*$ and $v\in V$.}
\end{equation*}

Take $V=\Af$ and let $\{f_i\}_{i\in\mathbb Z}$ be the dual basis relative to $\{t^i\}_{i\in\mathbb Z}$. Then on $V^*$ we have, for all $i, j\in\ZZ$, 
\begin{equation*}
(y\cdot f_i)(t^j)=f_i(\iota(y)\cdot t^j)=f_i(x\cdot t^j)=f_i(t^{j+1})=\delta_{i,j+1}=\delta_{i-1,j}=f_{i-1}(t^j). 
\end{equation*}
Thus, $y\cdot f_i=f_{i-1}$. Similarly, $x\cdot f_i=\mu(i+1)f_{i+1}$ and $h\cdot f_i=\lambda(i)f_i$.

Let $\Bf$ be the submodule of $\seq{\Af}^*$ spanned by $\{f_i\}_{i\in\mathbb Z}$. The module $\Bf$ is the so-called \textit{finite dual} of $\Af$. For convenience of notation, we can identify $\Bf$ with $\FF[t^{\pm1}]$ (so that $f_i$ corresponds to $t^i$). Then the action of $\hqfg$ on $\Bf$ is given as follows: 
\begin{equation}\label{E:reps:def:Bf}
ht^i=\lambda(i)t^i, \quad xt^i=\mu(i+1)t^{i+1}, \quad yt^i=t^{i-1}, \quad \mbox{for all $i\in\ZZ$.}
\end{equation}

\subsection{Finite-dimensional simple $\hqfg$-modules}\label{SS:irreps:fdiml}

We start out by constructing finite-dimensional simple $\hqfg$-modules as quotients of the modules $\Af$ and $\Bf$.

\begin{lemma}\label{L:irreps:fdiml:quotAB}
Assume that $q\neq 0$. Let $\lambda\in \Sf$, $\mu\in \Tf$ such that $|\lambda|, |\mu|\geq 1$. For any $\gamma\in\FF^*$,
the subspace $\FF[t^{\pm1}](t^{|\lambda| |\mu|}-\gamma)$ is a submodule of both $\Af$ and $\Bf$ and the corresponding factor modules
\begin{equation*}
\Af/\FF[t^{\pm1}](t^{|\lambda| |\mu|}-\gamma) \quad\mbox{and}\quad \Bf/\FF[t^{\pm1}](t^{|\lambda| |\mu|}-\gamma) 
\end{equation*}
are simple. 
\end{lemma}
\begin{proof}
We carry out the proof for $\Af$; the case of the dual module $\Bf$ is symmetric. 

Let $M=\FF[t^{\pm1}](t^{|\lambda| |\mu|}-\gamma t^0)$. Then $M$ has basis $\pb{P_j\mid j\in\ZZ},$ where $P_j=t^{j+|\lambda| |\mu|}-\gamma t^j$. We have 
\begin{align*}
x P_j&=P_{j+1},\\ 
h P_j&=\lambda(j+|\lambda| |\mu|)t^{j+|\lambda| |\mu|}-\gamma\lambda(j)t^j=\lambda(j)P_j \qquad \mbox{(by the definition of $|\lambda|$),}\\ 
y P_j&=\mu(j+|\lambda| |\mu|)t^{j-1+|\lambda| |\mu|}-\gamma\mu(j)t^{j-1}=\mu(j)P_{j-1} \qquad \mbox{(by Lemma~\ref{L:irreps:Verma:periodicl}).}
\end{align*}
So $M$ is indeed a submodule of $A_{\mathcal H_q(f,g)}(\lambda,\mu)$. Denote the quotient module by 
\begin{equation*}
\overline{\A}=\Af/M 
\end{equation*}
and set $T_i=t^i+M\in\overline{\A}$, for all $i\in\ZZ$. Then $\{T_0,\ldots,T_{|\lambda| |\mu|-1}\}$ is a basis of $\overline{\A}$ and it can be readily verified by induction on $|n|$ that $T_{i+n|\lambda| |\mu|}=\gamma^n T_i$, for all $i, n\in\ZZ$. The action on $\overline{\A}$ is given by 
\begin{equation*}
xT_j=T_{j+1}, \quad  hT_j=\lambda(j)T_{j}, \quad  yT_j=\mu(j)T_{j-1}, 
\end{equation*}
with $xT_{|\lambda| |\mu|-1}=\gamma T_0$ and $yT_0=\mu(0)\gamma^{-1}T_{|\lambda| |\mu|-1}$. 

Let $W\subseteq\overline{\A}$ be a nonzero submodule. Since the action of $h$ is diagonalizable on $\overline{\A}$, it must be so also on $W$, so $W$ contains some eigenvector for $h$, say $0\neq w\in W$. The eigenvalues of $h$ on $\overline{\A}$ are precisely $\lambda(0),\lambda(1),\ldots,\lambda(|\lambda|-1)$, so we must have $h w=\lambda(j)w$, for some $0\leq j<|\lambda|$. Since the $\lambda(j)$-eigenspace for $h$ on $\overline{\A}$ is $W_j=\spann_\FF\pb{T_{j+i|\lambda|}\mid 0\leq i<|\mu|}$, we have $w\in W_j$. Moreover, as $(xy) T_k=\mu(k)T_k$, for all $k$, it follows that $xy$ is also diagonalizable on $\overline{\A}$, and in particular on the nonzero subspace $W\cap W_j$. The eigenvalues for $xy$ on $W_j$ are $\mu(j), \mu(j+|\lambda|), \ldots, \mu(j+(|\mu|-1)|\lambda|)$, which are all distinct, by Lemma~\ref{L:irreps:Verma:periodicl}. Thus, the corresponding eigenspaces are of the form $\FF T_{j+i|\lambda|}$. In particular, there is $k$ such that $T_k\in W$. Then, using $xT_k=T_{k+1}$ and $xT_{|\lambda| |\mu|-1}=\gamma T_0$, with $\gamma\in\FF^*$, we conclude that $W=\overline{\A}$, which proves the simplicity of $\overline{\A}$.
\end{proof}

Next, we characterize the modules just obtained above.

%: x bijective

\begin{prop}\label{P:irreps:fdiml:xbijective}
Assume $\FF=\overline\FF$ and $q\neq 0$.
Let $V$ be a simple $\hqfg$-module with $\dim_\FF V=n$ and such that $x^nV\neq 0$. Then there are $\gamma\in\FF^*$, $\lambda\in \Sf$ and $\mu\in \Tf$ with $|\lambda|, |\mu|\geq 1$, $n=|\lambda||\mu|$ and
\begin{equation*}
V\simeq \Af/\FF[t^{\pm1}](t^{|\lambda||\mu|}-\gamma).
\end{equation*}
\end{prop}
\begin{proof}
We will break up the proof of this fundamental result into several steps. 
\begin{enumerate}[start=1,label={\it Step~\arabic*.},ref={\it Step~\arabic*},wide =0\parindent, leftmargin = \parindent]
\item $x$ acts bijectively on $V$.\\[5pt]
\noindent
We have the chain $x^nV\subseteq x^{n-1}V\subseteq \cdots \subseteq xV\subseteq V.$ If $x^iV\varsubsetneq x^{i-1}V$ for all $1\le i\le n$, then $\dim_\FF x^iV\leq n-i$, so $\dim_\FF x^nV=0$, which is a contradiction. So there is $1\le i\le n$ such that $x^iV=x^{i-1}V$ and thus $x^jV=x^{i-1}V$ for every $j\geq i-1$. Let $W=\cap_{j\ge0}x^jV=x^{i-1}V=x^nV\neq0$. Then $xW=x^{n+1}V=x^nV=W$, showing that $W$ is stable under the action of $x$ and that $x:W\to W$ is surjective. Since $\dim_\FF W<\infty$, this map is bijective, which means that $x$ acts bijectively on $W$. On the other hand, we have $hW=hx^nV=x^n\sigma^n(h)V\subseteq x^nV=W$, so $W$ is stable under the action of $h$. Finally, using Lemma \ref{L:basic:conformal} we have 
\begin{equation*}
yW=y(xW)=yx^{n+1}V=x^n\zeta V\subseteq x^nV=W, 
\end{equation*}
for some $\zeta\in\hqfg$. The preceeding shows that $W$ is a nonzero submodule of $V$. By the simplicity of $V$, we must have $W=V$ and thus $x$ acts bijectively on $V$.\label{P:irreps:fdiml:xbijective:st1}
\item There are $\ell\geq 1$ and $v_0\in V\setminus\{0\}$ such that $v_0$ is a common eigenvector for $h$ and for $xy$, with $h$-eigenvalue equal to $\alpha\in\FF$, where $f^{[\ell]}(\alpha)=\alpha$ and $\ell$ is minimal among positive integers with this property.\\[5pt]
\noindent
Since $h$ and $xy$ commute and $\FF=\overline\FF$, they have a common eigenvector on the finite-dimensional space $V$, say $v\in V\setminus\{0\}$. Moreover, by \ref{P:irreps:fdiml:xbijective:st1} and Lemma~\ref{L:weightmods}, $x^k v$ is still a common eigenvector for $h$ and $xy$, for all $k\geq 0$. Then the $h$-eigenvalue of $x^k v$ is $f^{[k]}(\alpha)$, where $\alpha$ is the $h$-eigenvalue of $v$. Since the number of eigenvalues must be finite and $x^k v\neq 0$ for all $k$, there are $i\geq 0$ and $\ell\geq1$ such that $f^{[i]}(\alpha)=f^{[i+\ell]}(\alpha)$, and we can assume that $\ell$ is minimal with this property. Replacing $v$ with $v_0=x^i v$, we can further assume without loss of generality that $i=0$. So, $h v_0=\alpha v_0$, $(xy) v_0\in\FF v_0$ and $f^{[\ell]}(\alpha)=\alpha$, with $\ell\geq1$ minimal.
\setcounter{tempcounter}{\value{enumi}}\addtocounter{tempcounter}{1}
\end{enumerate}

Henceforth, fix $\ell\geq 1$, $\alpha\in\FF$ and $v_0\in V\setminus\{0\}$ as given above.

\begin{enumerate}[start=\value{tempcounter},label={\it Step~\arabic*.},ref={\it Step~\arabic*},wide =0\parindent, leftmargin = \parindent]
\item Define $\lambda:\ZZ\longrightarrow\FF$ by $\lambda(i)=f^{[j]}(\alpha)$, where $0\leq j<\ell$ and $i\equiv j\ \modd{\ell}$. Then $\lambda\in\Sf$, $|\lambda|=\ell$ and $hx^iv_0=\lambda(i)x^iv_0$, for all $i\geq 0$.\\[5pt]
\noindent
We need to show that $f(\lambda(i))=\lambda(i+1)$ for all $i\in\ZZ$. Write $i=k\ell+j$ with $j$ as above and $k\in\ZZ$. If $j<\ell-1$, then $\lambda(i+1)=f^{[j+1]}(\alpha)=f(f^{[j]}(\alpha))=f(\lambda(i))$.
On the other hand, if $j=\ell-1$, then $i+1=(k+1)\ell$, so $\lambda(i+1)=f^{[0]}(\alpha)=\alpha$ and $f(\lambda(i))=f(f^{[\ell-1]}(\alpha))=f^{[\ell]}(\alpha)=\alpha$. Hence $\lambda\in \Sf$ and, by the definition of $\lambda$, $|\lambda|=\ell$. Finally, if $i\geq 0$ then $k\geq 0$ and $hx^iv_0=f^{[i]}(\alpha)x^i v_0=f^{[j]}(f^{[k\ell]}(\alpha))x^i v_0=f^{[j]}(\alpha)x^i v_0=\lambda(i)x^i v_0$. \label{P:irreps:fdiml:xbijective:st3}
\item $V=\spann_\FF\pb{x^kv_0\mid k\geq0}$.\\[5pt]
\noindent
Let $V'=\spann_\FF\pb{x^kv_0\mid k\geq0}$. As $xV'\subseteq V'$ and $x$ acts injectively on this finite-dimensional space, we have $xV'=V'$. Moreover, for any $k\geq0$, $x^kv_0$ is a common eigenvector for $h$ and $xy$, thus $h V'\subseteq V'$ and $y V'=yx V'=q(xy)V'+g(h)V'\subseteq V'$. So $V'$ is a nonzero submodule of $V$, hence $V=V'$, by the simplicity of $V$. 
\item There are $m\geq 1$ and $\omega_0\in V\setminus\{0\}$ such that $h\omega_0=\alpha\omega_0$, $(xy)\omega_0=\beta\omega_0$ and $x^{m\ell}\omega_0=\gamma\omega_0$, for some $\beta\in\FF$ and $\gamma\in\FF^*$. The integer, $m\geq 1$ is minimal such that $\mu_0=\mu_{m\ell}$, where $\mu_{k\ell}$ is defined by \eqref{E:irreps:fdiml:xbijective:defmusub} below.\\[5pt]
\noindent
Let $X=\spann_\FF\pb{x^{k\ell}v_0\mid k\geq 0}\subseteq V$. Notice that $x^{k\ell}v_0\neq0$ and $hx^{k\ell}v_0=\lambda(k\ell)x^{k\ell}v_0=\alpha x^{k\ell}v_0$, for all $k\geq 0$. Let $\beta$ be the $xy$-eigenvalue of $v_0$. Then, by Lemma~\ref{L:weightmods}, $(xy)x^{k\ell}v_0=\mu_{k\ell}x^{k\ell}v_0$, where
\begin{align}\label{E:irreps:fdiml:xbijective:defmusub}
\begin{split}
\mu_{k\ell}&=q^{k\ell}\beta+\sum\limits_{i=0}^{\ell-1}\sum\limits_{j=0}^{k-1}q^{i+j\ell}g\seq{f^{[(k-j)\ell-1-i]}(\alpha)}\\
&=q^{k\ell}\beta+\sum\limits_{i=0}^{\ell-1}q^{i}g\seq{f^{[\ell-1-i]}(\alpha)}\sum\limits_{j=0}^{k-1}q^{j\ell}\\
&=q^{k\ell}\beta+\Xi [k]_{q^\ell},
\end{split}
\end{align}
using the notation $\Xi=\sum\limits_{i=0}^{\ell-1}q^{i}g\seq{f^{[\ell-1-i]}(\alpha)}$ and $[k]_{q^\ell}=\sum\limits_{j=0}^{k-1}q^{j\ell}$.

Since $X$ is finite dimensional, there are $k\geq 0$ and $m\geq 1$ such that $\mu_{k\ell}=\mu_{(k+m)\ell}$ and we can assume that $m\geq 1$ is minimal with this property. As before, without loss of generality, we can also assume that $k=0$, so $\beta=\mu_0=\mu_{m\ell}$. Then, for $k\geq 0$, we have
\begin{align*}
\mu_{(k+1)m\ell}&= q^{(k+1)m\ell}\beta+\Xi [(k+1)m]_{q^\ell}\\
&=q^{km\ell+m\ell}\beta+\Xi \seq{[km]_{q^\ell}+q^{km\ell}[m]_{q^\ell}}\\
%&=q^{km\ell}\seq{q^{m\ell}\beta+\Xi [m]_{q^\ell}}+\Xi [km]_{q^\ell}\\
&=q^{km\ell}\mu_{m\ell}+\Xi [km]_{q^\ell}\\
&=q^{km\ell}\beta+\Xi [km]_{q^\ell}\\
&=\mu_{km\ell}
\end{align*}
and by induction we conclude that $\mu_{km\ell}=\beta$, for all $k\geq 0$.

Let $X'=\spann_\FF\pb{x^{km\ell}v_0\mid k\geq 0 }$. For any $\omega \in X'$ we have $h\omega=\alpha\omega$ and $(xy)\omega=\beta\omega$. Moreover, $x^{m\ell} X'\subseteq X'$, so $x^{m\ell}$ has some eigenvector $\omega_0$ on $X'$, say with eigenvalue $\gamma$. As $x$ acts bijectively on $V$, then $\gamma\in\FF^*$. So $\omega_0\in V\setminus\pb{0}$, $h\omega_0=\alpha\omega_0$, $(xy)\omega_0=\beta\omega_0$ and $x^{m\ell}\omega_0=\gamma\omega_0$.\label{P:irreps:fdiml:xbijective:st5}
\setcounter{tempcounter}{\value{enumi}}\addtocounter{tempcounter}{1}
\end{enumerate}

Henceforth, fix also $m\geq 1$, $\beta\in\FF$, $\gamma\in\FF^*$ and $\omega_0\in V\setminus\{0\}$ as given above. Moreover, define $\omega_j=x^j\omega_0$, for $0< j<m\ell$.

\begin{enumerate}[start=\value{tempcounter},label={\it Step~\arabic*.},ref={\it Step~\arabic*},wide =0\parindent, leftmargin = \parindent]
\item $V=\spann_\FF\pb{\omega_0,\omega_1,\ldots,\omega_{m\ell-1}}$.\\[5pt]
\noindent
We have $x\omega_j=\omega_{j+1}$ for $0\leq j\leq m\ell -2$ and $x\omega_{m\ell-1}=x^{m\ell}\omega_0=\gamma\omega_0$; $h\omega_j=f^{[j]}(\alpha)\omega_j=\lambda(j)\omega_j$. Regarding the action of $y$, we have $xy\omega_0=\beta\omega_0=\beta\gamma^{-1}x\omega_{m\ell-1}$ and since the action of $x$ is injective, it follows that $y\omega_0=\beta\gamma^{-1}\omega_{m\ell-1}$. Similarly, for $j\geq 1$, using Lemma~\ref{L:weightmods},
\begin{align*}
xy\omega_j=xyx^j\omega_0=\seq{q^j\beta+\sum\limits_{i=0}^{j-1}q^i g\seq{f^{[j-i-1]}(\alpha)}}x\omega_{j-1},
\end{align*}
so 
\begin{align*}
y\omega_j=\seq{q^j\beta+\sum\limits_{i=0}^{j-1}q^i g\seq{f^{[j-i-1]}(\alpha)}}\omega_{j-1}.
\end{align*}
In particular, $\spann_\FF\pb{\omega_0,\omega_1,\ldots,\omega_{m\ell-1}}$ is a nonzero submodule of $V$, hence equal to $V$.\label{P:irreps:fdiml:xbijective:st6}
\item Define $\mu:\ZZ\longrightarrow\FF$ by 
\begin{align*}
\mu(i)=q^j\beta+\sum\limits_{k=0}^{j-1}q^k g\seq{f^{[j-k-1]}(\alpha)}=q^j\beta+\sum\limits_{k=0}^{j-1}q^k g(\lambda(j-k-1)),
\end{align*}
where $0\leq j<m\ell$ and $i\equiv j\ \modd{m\ell}$. Then $\mu\in\Tf$ and $|\mu|=m$.\\[5pt]
\noindent
The proof that $\mu(i+1)=q\mu(i)+g(\lambda(i))$, for all $i\in\ZZ$, is analogous to the proof in \ref{P:irreps:fdiml:xbijective:st3} that $\lambda\in\Sf$, using also the equality $\mu_{m\ell}=\mu_0=\beta=\mu(0)$, from \ref{P:irreps:fdiml:xbijective:st5}, while $|\mu|=m$ by the definition of $m$.
\item For $i\in\ZZ$, write $i=k(m\ell)+j$ with $0\leq j<m\ell$ and set $\omega_i=\gamma^k\omega_j$. Then, 
\begin{equation}\label{E:reps:action:omegas}
x\omega_i=\omega_{i+1},\quad y\omega_i=\mu(i)\omega_{i-1}\quad \text{and}\quad h\omega_i=\lambda(i)\omega_i,\quad \text{for all $i\in\mathbb Z$.}
\end{equation}

\noindent
Write $i=k(m\ell)+j$, as above. Using the relations established in \ref{P:irreps:fdiml:xbijective:st6}, the definition of $\mu$ and $|\lambda|=\ell$, it is immediate to verify that $h\omega_i=\lambda(i)\omega_i$ for all $i$; that $x\omega_i=\omega_{i+1}$, as long as $j\leq m\ell-2$; and that $y\omega_i=\mu(i)\omega_{i-1}$, as long as $j\geq 1$.

Now suppose $j=m\ell-1$. Then $i+1=(k+1)(m\ell)$, so $x\omega_i=\gamma^k x\omega_{m\ell-1}=\gamma^{k+1} \omega_{0}=\omega_{i+1}$. Similarly, if $j=0$, then $i-1=(k-1)(m\ell)+m\ell-1$ and $y\omega_i=\gamma^k y\omega_0=\beta\gamma^{k-1} \omega_{m\ell-1}=\beta\omega_{i-1}=\mu(0)\omega_{i-1}=\mu(i)\omega_{i-1}$.
\item $V$ is isomorphic to $\Af/\FF[t^{\pm1}](t^{|\lambda||\mu|}-\gamma)$. In particular, $n=\dim_\FF V=|\lambda||\mu|$.\\[5pt]
\noindent
Consider the linear map $\phi:\Af\longrightarrow V$ defined by $\phi(t^i)=\omega_i$, for all $i\in\ZZ$. Comparing relations \eqref{E:reps:def:Af} and \eqref{E:reps:action:omegas}, we see that $\phi$ is an $\hqfg$-module homomorphism and, by \ref{P:irreps:fdiml:xbijective:st6}, $\phi$ is surjective. Moreover, by the definition of $\omega_i$,  
\begin{equation*}
\phi(t^{i+m\ell}-\gamma t^i)=\omega_{i+m\ell}-\gamma\omega_i=\gamma\omega_{i}-\gamma\omega_i=0, \quad \text{for all $i\in\ZZ$.}
\end{equation*}
Hence, $\phi$ factors through $\FF[t^{\pm1}](t^{|\lambda||\mu|}-\gamma)$, which by Lemma~\ref{L:irreps:fdiml:quotAB} is a maximal submodule of $\Af$. Thus $\phi$ induces a surjective homomorphism
\begin{equation*}
\overline\phi:\Af/\FF[t^{\pm1}](t^{|\lambda||\mu|}-\gamma)\longrightarrow V
\end{equation*}
and by the maximality of $\FF[t^{\pm1}](t^{|\lambda||\mu|}-\gamma)$ in $\Af$, $\overline\phi$ is injective, thence an isomorphism.
\end{enumerate}
\end{proof}	

The analogous result for $n$-dimensional simple modules $V$ such that $y^nV\neq 0$ uses the dual modules $\Bf$. This result below follows easily from Proposition~\ref{P:irreps:fdiml:xbijective} and the natural $\hqfg$-module isomorphism $V\simeq V^{**}$, where the action on the dual space is defined in terms of the order $2$ anti-automorphism $\iota$, as explained in Subsection~\ref{SS:irreps:Verma}.

%: y bijective

 \begin{prop}\label{P:irreps:fdiml:ybijective}
Assume $\FF=\overline\FF$ and $q\neq 0$.
Let $V$ be a simple $\hqfg$-module with $\dim_\FF V=n$ and such that $y^nV\neq 0$. Then there are $\gamma\in\FF^*$, $\lambda\in \Sf$ and $\mu\in \Tf$ with $|\lambda|, |\mu|\geq 1$, $n=|\lambda||\mu|$ and
\begin{equation*}
V\simeq \Bf/\FF[t^{\pm1}](t^{|\lambda||\mu|}-\gamma).
\end{equation*}
\end{prop}

\smallskip

It remains to study those finite-dimensional simple modules on which both $x$ and $y$ act nilpotently. Fix $\alpha\in\FF$ and set $\nu_\alpha:\ZZ_{\geq 0}\longrightarrow\FF$ so that
\begin{equation}\label{E:reps:def:nualpha}
\nu_\alpha(i)=\sum\limits_{j=0}^{i-1}q^jg(f^{[i-1-j]}(\alpha)), \quad\text{for all $i\geq 0$.} 
\end{equation}
Define the $\hqfg$-module $\Cf=\FF[t]$ with action
\begin{equation}\label{E:reps:def:Cf}
ht^i=f^{[i]}(\alpha) t^i, \quad xt^i=t^{i+1}, \quad yt^{i}=\nu_\alpha(i) t^{i-1}, \quad \text{for all $i\geq 0$,}
\end{equation}
adopting the convention that $yt^{0}=0$. A routine check shows that \eqref{E:reps:def:Cf} does define an $\hqfg$-module structure on $\FF[t]$.

\begin{lemma}\label{L:irreps:fdiml:Cfsimple}
Let $\alpha\in\FF$ and $\nu_\alpha$ be given by~\eqref{E:reps:def:nualpha}. Suppose that $\nu_\alpha(n)=0$, for some $n\geq 1$. Then $\FF[t]t^n$ is a submodule of $\Cf$ and the quotient module $\Cf/\FF[t]t^n$ is simple if and only if $\nu_\alpha(1)\cdots\nu_\alpha(n-1)\neq0$.
\end{lemma}
\begin{proof}
It is clear that $\FF[t]t^n$ is closed under the actions of $x$ and $h$. Also, 
$y t^n=\nu_\alpha(n)t^{n-1}=0$, so $\FF[t]t^n$ is closed under the action of $y$ as well, and is thus a submodule of $\Cf$.

The quotient module has basis $\pb{\omega_0,\ldots,\omega_{n-1}}$, where $\omega_i=t^i+\FF[t]t^n$ and the action is given by
\begin{gather}\label{E:irreps:fdiml:Cfsimple:basis}
\begin{split}
x\omega_i=\omega_{i+1}, \ \text{if $i<n-1$}, \quad x\omega_{n-1}=0,\\
h\omega_i=f^{[i]}(\alpha)\omega_i,\quad y\omega_i=\nu_\alpha(i)\omega_{i-1}, \ \text{if $i>0$}, \quad y\omega_0=0. 
\end{split}
\end{gather}
Suppose $\nu_\alpha(i)=0$ for some $1\leq i\leq n-1$. Then $\spann_\FF\pb{\omega_i,\omega_{i+1},\ldots,\omega_{n-1}}$ is a proper nonzero submodule, so $\Cf/\FF[t]t^n$ is not simple in this case.

Conversely, suppose that $\nu_\alpha(i)\neq0$ for all $1\leq i\leq n-1$. Let $W\subseteq \Cf/\FF[t]t^n$ be a nonzero submodule. Then applying $x$ enough times to a nonzero element of $W$, we deduce that $\omega_{n-1}\in W$. Then $\nu_\alpha(1)\cdots\nu_\alpha(n-1)\omega_0=y^{n-1}\omega_{n-1}\in W$. Since, by hypothesis, $\nu_\alpha(1)\cdots\nu_\alpha(n-1)\neq 0$, we get $\omega_0\in W$ and, $\omega_0$ being a generator of $\Cf/\FF[t]t^n$, it follows that $W=\Cf/\FF[t]t^n$. This concludes the proof.
\end{proof}

%: x and y nilpotent

Now we characterize the finite-dimensional simple $\hqfg$-modules on which both $x$ and $y$ act nilpotently.

\begin{prop}\label{P:irreps:fdiml:xynilpotent}
Assume $\FF=\overline\FF$.
Let $V$ be a simple $\hqfg$-module with $\dim_\FF V=n$ and $x^nV=0=y^nV$. Then there is $\alpha\in\FF$ such that $\nu_\alpha(n)=0$ and $V\simeq \Cf/\FF[t]t^n$, where $\nu_\alpha$ is given by~\eqref{E:reps:def:nualpha}.
\end{prop}
\begin{proof}
Let $X=\pb{v\in V\mid yv=0}$. Then $X\neq\pb{0}$ because $y^nV=0$. Also, $yhX=f(h)yX=0$, so $hX\subseteq X$. Thus, $h$ has some eigenvector $0\neq\omega_0$ in $X$, say $h\omega_0=\alpha \omega_0$, with $\alpha\in\FF$. By the simplicity of $V$, we know that $V=\hqfg\omega_0$. But $\pb{x^ih^jy^k\mid i,j,k\geq 0}$ is a basis of $\hqfg$ and, by hypothesis, $x^n\omega_0=0$, so $V=\spann_\FF\pb{\omega_0, \omega_1,\ldots ,\omega_{n-1}}$, where $\omega_i=x^i\omega_0$, for $0<i<n$. Moreover, since $\dim_\FF V=n$, these vectors must form a basis of $V$.

Then, using Lemma~\ref{L:basic:conformal}, we have
\begin{gather}\label{E:irreps:fdiml:xynilpotent}
\begin{split}
x\omega_i=\omega_{i+1}, \ \text{if $i<n-1$}, \quad x\omega_{n-1}=0,\\
h\omega_i=f^{[i]}(\alpha)\omega_i,\quad y\omega_i=\nu_\alpha(i)\omega_{i-1}, \ \text{if $i>0$}, \quad y\omega_0=0. 
\end{split}
\end{gather}
In addition, $0=yx^n\omega_0=\nu_\alpha(n)\omega_{n-1}$, and hence $\nu_\alpha(n)=0$. Now, by comparing the relations \eqref{E:irreps:fdiml:Cfsimple:basis} and \eqref{E:irreps:fdiml:xynilpotent}, the isomorphism $V\simeq \Cf/\FF[t]t^n$ becomes clear.
\end{proof}

%: main Thm f-dim irreps

We can finally prove our main result which classifies, up to isomorphism, all finite-dimensional simple $\hqfg$-modules.

\begin{thm}\label{T:irreps:fdiml:all}
Assume $\FF=\overline\FF$ and $q\neq 0$. Then any simple $n$-dimensional $\hqfg$-module is isomorphic to exactly one of the following simple modules:
\begin{enumerate}[label=(\alph*)]
%x acts invertibly
\item $\Af/\FF[t^{\pm1}](t^{|\lambda||\mu|}-\gamma)$, for some $\lambda\in\Sf$, $\mu\in\Tf$ and $\gamma\in\FF^*$ such that $n=|\lambda||\mu|$.\label{T:irreps:fdiml:all:a}
%x does not act invertibly but y acts invertibly
\item $\Bf/\FF[t^{\pm1}](t^{|\lambda||\mu|}-\gamma)$, for some $\lambda\in\Sf$, $\mu\in\Tf$ and $\gamma\in\FF^*$ such that $n=|\lambda||\mu|$ and $\mu(i)=0$ for some $0\leq i<|\lambda||\mu|$.\label{T:irreps:fdiml:all:b}
%neither x not y act invertibly
\item $\Cf/\FF[t]t^n$, for some $\alpha\in\FF$ such that $\nu_\alpha(n)=0$ and $\nu_\alpha(i)\neq0$ for all $1\leq i\leq n-1$.\label{T:irreps:fdiml:all:c}
\end{enumerate}
\end{thm}

\begin{proof}
The modules in \ref{T:irreps:fdiml:all:a}--\ref{T:irreps:fdiml:all:c} are simple by Lemma~\ref{L:irreps:fdiml:quotAB} and Lemma~\ref{L:irreps:fdiml:Cfsimple}.

Conversely, let $V$ be a simple $n$-dimensional $\hqfg$-module. If $x^nV\neq 0$ then, by Proposition~\ref{P:irreps:fdiml:xbijective}, $V$ is isomorphic to a module of the form given in \ref{T:irreps:fdiml:all:a}.
Suppose now that $x^nV=0$. If $y^nV=0$ also holds, then by Proposition~\ref{P:irreps:fdiml:xynilpotent} and Lemma~\ref{L:irreps:fdiml:Cfsimple}, $V$ is isomorphic to a module of the form given in \ref{T:irreps:fdiml:all:c}.
Otherwise, we have $x^nV=0$ and $y^nV\neq0$. Then, by Proposition~\ref{P:irreps:fdiml:ybijective}, $V$ is isomorphic to a module of the form given in \ref{T:irreps:fdiml:all:b}. Moreover, from \eqref{E:reps:def:Bf}, $x^n=x^{|\lambda||\mu|}$ acts on $V$ by the scalar $\gamma\prod_{j=0}^{m\ell-1}\mu(j)$, so $\mu(j)=0$ for some $0\leq j<m\ell$.

Finally, to see that modules from distinct parts of \ref{T:irreps:fdiml:all:a}--\ref{T:irreps:fdiml:all:c} are non-isomorphic, we just need to observe that indeed $x^nV\neq 0$ for all modules from \ref{T:irreps:fdiml:all:a}, since $x^n$ acts by multiplication by $\gamma$; $x^nV=0$ for all modules from \ref{T:irreps:fdiml:all:b} (by the previous paragraph) and \ref{T:irreps:fdiml:all:c}; $y^nV\neq0$ for all modules from \ref{T:irreps:fdiml:all:b}, since $y^n=y^{|\lambda||\mu|}$ acts on such a $V$ by the scalar $\gamma^{-1}$; and $y^nV=0$ for all modules from \ref{T:irreps:fdiml:all:c}.
\end{proof}

\begin{exam}
Assume that $q\neq 0$, $f=h^a$ and $g=h^b$, with $a, b\geq 1$. 

The finite-dimensional simple modules of the form $\Cf/\FF[t]t^n$ depend on the existence of $\alpha$ such that $\nu_\alpha(n)=0$ and $\nu_\alpha(i)\neq0$ for all $1\leq i\leq n-1$, see~\eqref{E:reps:def:nualpha}. In case $n=1$, and since $b\geq 1$, the only possibility is the \textit{trivial} module $\Cf[0]/\FF[t]t$ on which all the generators act as $0$. If we define the polynomial $\nu_t(n)=\sum_{j=0}^{n-1}q^j t^{a^{n-1-j}b}$ of degree $a^{n-1}b$ in the variable $t$ then, for $n>1$, these modules are determined solely by the roots of $\nu_t(m)$, for $m\leq n$, and that study depends heavily on $q$ and the base field $\FF$. 

Hence, we focus on the simple quotients of the modules $\Af$ and $\Bf$, and it suffices to determine the possibilities for $\lambda\in\Sf$ and $\mu\in\Tf$, with $|\lambda|, |\mu|\geq 1$.

For $\ell, m\geq 0$, define
\begin{align*}
\Sf^\ell=\pb{\lambda\in\Sf\mid |\lambda|=\ell}\quad\text{and}\quad \Tf^m= \pb{\mu\in\Tf\mid |\mu|=m}.
\end{align*}
Let $\lambda\in\Sf$ and set $\alpha=\lambda(0)$. Then, $\lambda(i)=\alpha^{a^i}$, for all $i\geq 0$. If either $\alpha=0$ or $a=1$, then $\lambda\equiv \alpha$ and $|\lambda|=1$. So assume that $\alpha\neq 0$ and $a\geq 2$. In this case, if $\alpha^{a^k -1}\neq 1$ for all $k\geq 1$ then $|\lambda|=0$. Otherwise, if $\ell=\min\pb{k\geq 1\mid \alpha^{a^k -1}=1}$, then $\alpha$ determines a unique $\lambda\in\Sf^\ell$ with $\lambda(i)=\alpha^{a^j}$, for all $i\in\ZZ$, where $0\leq j<\ell$ and $i\equiv j\modd{\ell}$.

Next we determine all possible $\mu\in\Tf$ with $|\mu|<\infty$. Recall that, by Lemma~\ref{L:irreps:Verma:mu}, $\mu$ is completely determined by $\alpha$ and $\mu(0)$. Notice also that $\Xi=\nu_\alpha(\ell)$, where $\Xi$ is the scalar appearing in~\eqref{E:irreps:fdiml:xbijective:defmusub}.

\begin{enumerate}[start=1,label={\it Case~\Alph*:},wide =0\parindent, leftmargin = \parindent]
\item $q$ is not a root of unity.

Then, by~\eqref{E:irreps:fdiml:xbijective:defmusub}, the only $\mu\in\Tf$ with $|\mu|<\infty$ has $|\mu|=1$ and $\mu(0)=\frac{\nu_\alpha(\ell)}{1-q^\ell}$.

\item $q$ is a root of unity. 
\begin{itemize}
\item Suppose $q^\ell=1$. Then either $\nu_\alpha(\ell)=0$, $\mu(0)=\beta\in\FF$ is arbitrary and $|\mu|=1$; or $\nu_\alpha(\ell)\neq 0$, $\mu(0)=\beta\in\FF$ is arbitrary, $\chara(\FF)=p>0$ and $|\mu|=p$.
\item Suppose $q^\ell\neq1$. Then either $\mu(0)=\frac{\nu_\alpha(\ell)}{1-q^\ell}$ and $|\mu|=1$; or otherwise $\mu(0)\in\FF\setminus\pb{\frac{\nu_\alpha(\ell)}{1-q^\ell}}$ is arbitrary and $|\mu|=|\langle q^\ell\rangle|$.
\end{itemize}
\end{enumerate}
\end{exam}

\subsection{Isomorphisms between finite-dimensional simple $\hqfg$-modules}\label{SS:irreps:isos}

In this subsection we study when two of the modules appearing in the classification given in Theorem~\ref{T:irreps:fdiml:all} are isomorphic. It follows from the proof of that result that one module from one of the items in that theorem cannot be isomorphic to one from any of the other items, but it is possible for two distinct ones of the same type to be isomorphic to each other. Also, removing the restriction in item \ref{T:irreps:fdiml:all:b} that $\mu(i)=0$ for some $0\leq i<|\lambda||\mu|$, it becomes possible for a simple module of the form given in \ref{T:irreps:fdiml:all:b} to be isomorphic to one of the form given in \ref{T:irreps:fdiml:all:a}.

\begin{prop}\label{P:irreps:isosA}
Assume $q\neq 0$ and let $\lambda, \lambda'\in\Sf$, $\mu\in\Tf$, $\mu'\in\Tf[q,g,\lambda']$ and $\gamma, \gamma'\in\FF^*$, with $|\lambda|, |\lambda'|, |\mu|, |\mu'|\geq 1$. Then
\begin{equation*}
\Af/\FF[t^{\pm1}](t^{|\lambda||\mu|}-\gamma)\simeq \Af[\lambda', \mu']/\FF[t^{\pm1}](t^{|\lambda'||\mu'|}-\gamma') 
\end{equation*}
if and only if $\gamma=\gamma'$ and there exist $0\leq i<|\lambda|$, $0\leq j<|\mu|$ such that 
\begin{equation}\label{E:irreps:isosA}
\lambda'(k)=\lambda(k+i)\quad \text{and}\quad \mu'(k)=\mu(k+i+j|\lambda|),\quad \text{for all $k\in\ZZ$.}
\end{equation}
\end{prop}

\begin{proof}
We start out the proof by establishing notation and recalling the properties of the simple modules involved.

Set $V=\Af/\FF[t^{\pm1}](t^{|\lambda||\mu|}-\gamma)$ and $V'=\Af[\lambda', \mu']/\FF[t^{\pm1}](t^{|\lambda'||\mu'|}-\gamma')$. Then $\dim_\FF V=|\lambda||\mu|$ and $\dim_\FF V'=|\lambda'||\mu'|$. Define $\omega_i=t^i+\FF[t^{\pm1}](t^{|\lambda||\mu|}-\gamma)\in V$, so that $\omega_{i+k|\lambda||\mu|}=\gamma^k\omega_i$, for all $i, k\in\ZZ$. Then $\pb{\omega_i}_{i\in I}$ is a basis of $V$ for all $I\subseteq \ZZ$ such that $I$ contains precisely one representative of each congruence class modulo $|\lambda||\mu|$. As in \eqref{E:reps:action:omegas}, for all $i\in\mathbb Z$,
\begin{equation*}
x\omega_i=\omega_{i+1},\quad y\omega_i=\mu(i)\omega_{i-1},\quad h\omega_i=\lambda(i)\omega_i\quad \text{and thus also}\quad (xy)\omega_i=\mu(i)\omega_i.
\end{equation*}
The elements $\omega_i'\in V'$ are defined similarly and have identical properties.

The action of $h$ on $V$ is diagonalizable with $|\lambda|$ distinct eigenvalues, $\lambda(0), \ldots, \lambda(|\lambda|-1)$, the eigenspace for $\lambda(i)$ being
\begin{equation*}
V_i=\spann_\FF\pb{\omega_{i+j|\lambda|}\mid 0\leq j<|\mu|}. 
\end{equation*}
Similarly,
\begin{equation*}
V_i'=\spann_\FF\pb{\omega_{i+j|\lambda'|}'\mid 0\leq j<|\mu'|} 
\end{equation*}
is the eigenspace for the eigenvalue $\lambda'(i)$ corresponding to the action of $h$ on $V'$.

Suppose that $\phi:V\longrightarrow V'$ is an isomorphism. Then in particular these modules have the same dimension, so $|\lambda||\mu|=|\lambda'||\mu'|$. The element $x^{|\lambda||\mu|}$ acts on $V$ via multiplication by $\gamma$ and on $V'$ via multiplication by $\gamma'$. Thus, $\gamma=\gamma'$. Moreover, since the number of distinct eigenvalues of $h$ must be the same on $V$ and on $V'$, we have $|\lambda|=|\lambda'|$ and consequently also $|\mu|=|\mu'|$.

For all $k$, $\phi(V_k)$ is an $h$-eigenspace in $V'$ with eigenvalue $\lambda(k)$. As $\phi$ is bijective, this must be in fact a full eigenspace and furthermore $\lambda(\ZZ)=\lambda'(\ZZ)$. In particular, there is  $0\leq i<|\lambda|$ such that $\lambda(i)=\lambda'(0)$ and $\phi(V_i)=V'_0$. Then, as $x^k V_j=V_{j+k}$ for all $j$, and likewise for $V'_j$, the equality
\begin{equation*}
V'_k=x^kV'_0=x^k \phi(V_i)=\phi(x^kV_i)=\phi(V_{k+i})
\end{equation*}
says that $\lambda'(k)=\lambda(k+i)$, for all $k\geq 0$. By the periodicity of $\lambda$ and $\lambda'$, this relation holds for all $k\in\ZZ$.

Now, $\omega'_0\in V'_0= \phi(V_i)$ and the $(xy)$-eigenvalue of $\omega'_0$ is $\mu'(0)$. On the other hand, the $(xy)$-eigenvalues on $\phi(V_i)$ are $\mu(i+j|\lambda|)$, for $0\leq j<|\mu|$, and by Lemma~\ref{L:irreps:Verma:periodicl}\ref{L:irreps:Verma:periodicl:c}, these are all distinct. So, there is one (unique) such $j$ satisfying $\mu'(0)=\mu(i+j|\lambda|)$ and $\phi(\omega_{i+j|\lambda|})=\alpha\omega'_0$, for some nonzero scalar $\alpha$. Then, as above, acting on both sides of this equality by $x^k$, we find that $\phi(\omega_{k+i+j|\lambda|})=\alpha\omega'_k$, from which follows that $\mu'(k)=\mu(k+i+j|\lambda|)$, first for all $k\geq 0$ and then for all $k\in\ZZ$, by the periodicity of $\mu$ and $\mu'$.

Conversely, if $\gamma=\gamma'$ and \eqref{E:irreps:isosA} holds for some $0\leq i<|\lambda|$ and $0\leq j<|\mu|$, then $|\lambda|=|\lambda'|$ and $|\mu|=|\mu'|$, which implies that $\dim_\FF V=\dim_\FF V'$. Define the linear map $\phi:V\longrightarrow V'$ on the basis $\pb{\omega_k}_{0\leq k<|\lambda||\mu|}$ of $V$ by $\omega_k\mapsto\omega'_{k-i-j|\lambda|}$. By an earlier remark, $\pb{\omega'_{k-i-j|\lambda|}}_{0\leq k<|\lambda||\mu|}$ is a basis of $V'$, so $\phi$ is bijective. To show that $\phi$ is an isomorphism, it suffices to verify that $a\phi(\omega_k)=\phi(a\omega_k)$ for all $k\in\ZZ$ and all $a\in\pb{x, h, y}$. In case $a=x$ this is immediate, in case $a=h$ this follows from \eqref{E:irreps:isosA} and the fact that $|\lambda|=|\lambda'|$ and in case $a=y$ this follows from \eqref{E:irreps:isosA}. So indeed $V\simeq V'$.
\end{proof}

\begin{remark}
The proof of Proposition~\ref{P:irreps:isosA} implies that 
\begin{equation*}
\mathrm{End}_{\hqfg}\seq{\Af/\FF[t^{\pm1}](t^{|\lambda||\mu|}-\gamma)}=\FF\, 1, 
\end{equation*}
where $1$ is the identity endomorphism. In case $\FF=\overline{\FF}$, this is just a consequence of Schur's Lemma, but we now know that this holds over an arbitrary field $\FF$.
\end{remark}

There is of course a dual result for the finite-dimensional simple quotients of $\Bf$.

\begin{prop}\label{P:irreps:isosB}
Assume $q\neq 0$ and let $\lambda, \lambda'\in\Sf$, $\mu\in\Tf$, $\mu'\in\Tf[q,g,\lambda']$ and $\gamma, \gamma'\in\FF^*$, with $|\lambda|, |\lambda'|, |\mu|, |\mu'|\geq 1$. Then
\begin{equation*}
\Bf/\FF[t^{\pm1}](t^{|\lambda||\mu|}-\gamma)\simeq \Bf[\lambda', \mu']/\FF[t^{\pm1}](t^{|\lambda'||\mu'|}-\gamma') 
\end{equation*}
if and only if $\gamma=\gamma'$ and there exist $0\leq i<|\lambda|$, $0\leq j<|\mu|$ such that \eqref{E:irreps:isosA} holds.
\end{prop}

Next, we tackle the finite-dimensional simple quotients of $\Cf$. Since $\Cf/\FF[t]t^n$, when defined, is $n$-dimensional, the only question is the unicity of $\alpha\in\FF$, which is easy to establish in case $\Cf/\FF[t]t^n$ is simple.

\begin{prop}\label{P:irreps:isosC}
Let $\alpha, \alpha'\in\FF$ and $n\geq 1$ be such that $\nu_\alpha(n)=0=\nu_{\alpha'}(n)$ and $\nu_\alpha(i), \nu_{\alpha'}(i)\neq0$, for all $1\leq i\leq n-1$. Then $\Cf/\FF[t]t^n\simeq \Cf[\alpha']/\FF[t]t^n$ if and only if $\alpha=\alpha'$.
\end{prop}
\begin{proof}
Since $\nu_\alpha(i)\neq0$, for all $1\leq i\leq n-1$, then by \eqref{E:irreps:fdiml:Cfsimple:basis}, $\pb{v\in\Cf/\FF[t]t^n \mid yv=0}$ is one dimensional and $h$ acts on this space by multiplication by $\alpha$. Any isomorphism from $\Cf/\FF[t]t^n$ to $\Cf[\alpha']/\FF[t]t^n$ sends the space $\pb{v\in\Cf/\FF[t]t^n \mid yv=0}$ to $\pb{v'\in\Cf[\alpha']/\FF[t]t^n \mid yv'=0}$ and $h$ acts on the latter by multiplication by $\alpha'$, it follows that $\Cf/\FF[t]t^n\simeq \Cf[\alpha']/\FF[t]t^n$ implies that $\alpha=\alpha'$.
\end{proof}

Finally, we see when a finite-dimensional simple quotient of $\Af$ is isomorphic to a quotient of $\Bf$.

\begin{prop}\label{P:irreps:isosAB}
Assume $q\neq 0$ and let $\lambda, \lambda'\in\Sf$, $\mu\in\Tf$, $\mu'\in\Tf[q,g,\lambda']$ and $\gamma, \gamma'\in\FF^*$, with $|\lambda|, |\lambda'|, |\mu|, |\mu'|\geq 1$. Then
\begin{equation*}
\Af/\FF[t^{\pm1}](t^{|\lambda||\mu|}-\gamma)\simeq \Bf[\lambda', \mu']/\FF[t^{\pm1}](t^{|\lambda'||\mu'|}-\gamma') 
\end{equation*}
if and only if $\gamma=\gamma'\prod_{j=0}^{|\lambda||\mu|-1}\mu(j)$ and there exist $0\leq i<|\lambda|$, $0\leq j<|\mu|$ such that \eqref{E:irreps:isosA} holds.
\end{prop}

\begin{proof}
Suppose that $\prod_{j=0}^{|\lambda||\mu|-1}\mu(j)\neq 0$. Set $V=\Af/\FF[t^{\pm1}](t^{|\lambda||\mu|}-\gamma)$, $n=|\lambda||\mu|$ and $\omega_i=t^i+\FF[t^{\pm1}](t^{n}-\gamma)\in V$, as in Proposition~\ref{P:irreps:isosA}. We define a new basis $\pb{\overline\omega_i}_{0\leq i<n}$ of $V$ by $\overline\omega_i=\seq{\prod_{j=0}^{i}\mu(j)^{-1}}\omega_i$, for all $0\leq i<n$. Then:
\begin{gather*}
h\overline\omega_i=\lambda(i)\overline\omega_i,\\
x\overline\omega_i=\mu(i+1)\overline\omega_{i+1}, \ \text{if $i<n-1$}, \quad x\overline\omega_{n-1}=\frac{\gamma}{\mu(1)\cdots\mu(n-1)}\overline\omega_0,\\
y\overline\omega_i=\overline\omega_{i-1}, \ \text{if $i>0$}, \quad y\overline\omega_0
=\frac{\mu(0)\cdots\mu(n-1)}{\gamma}\overline\omega_{n-1}. 
\end{gather*}
Comparing the above with \eqref{E:reps:def:Bf} we see that 
\begin{equation*}
V\simeq \Bf/\FF[t^{\pm1}](t^{n}-\tilde\gamma), \quad \text{with $\tilde\gamma=\frac{\gamma}{\mu(0)\cdots\mu(n-1)}$.}
\end{equation*}

So suppose that $V\simeq V'$, where $V'=\Bf[\lambda', \mu']/\FF[t^{\pm1}](t^{|\lambda'||\mu'|}-\gamma')$. Then, in particular, $|\lambda'||\mu'|=n$. Notice that $y^n$ acts on $V$ as $\frac{\mu(0)\cdots\mu(n-1)}{\gamma}$ and on $V'$ as $\frac{1}{\gamma'}$. So $\gamma=\gamma'\prod_{j=0}^{n-1}\mu(j)$. Since $\gamma\neq 0$, it follows also that $\prod_{j=0}^{n-1}\mu(j)\neq 0$. By the first paragraph of the proof, we deduce that $V\simeq \Bf/\FF[t^{\pm1}](t^{n}-\gamma')$. So $\Bf/\FF[t^{\pm1}](t^{n}-\gamma')\simeq\Bf[\lambda', \mu']/\FF[t^{\pm1}](t^{|\lambda'||\mu'|}-\gamma')$ and the desired relations between $\lambda, \lambda',\mu, \mu'$ are given by Proposition~\ref{P:irreps:isosB}.

Conversely, suppose that $\gamma=\gamma'\prod_{j=0}^{n-1}\mu(j)$ and that \eqref{E:irreps:isosA} holds, for some $i, j$. Then, again we have that $|\lambda'||\mu'|=|\lambda||\mu|=n$ and $\prod_{j=0}^{n-1}\mu(j)\neq 0$. Therefore, 
\begin{equation*}
V\simeq \Bf/\FF[t^{\pm1}](t^{n}-\gamma')\simeq  \Bf[\lambda', \mu']/\FF[t^{\pm1}](t^{n}-\gamma'),
\end{equation*}
by Proposition~\ref{P:irreps:isosB}.
\end{proof}

\section{Concluding remarks}\label{S:concluding}

The class of quantum generalized Heisenberg algebras seems interesting and worth studying further. In a future work~\cite{LR20pr}, we determine when a quantum generalized Heisenberg algebra is noetherian, tackle the isomorphism problem for this class and also describe the associated groups of automorphisms. There are many more features of this class of algebras which deserve being investigated. We conclude the paper with a series of questions for further study.

\begin{itemize}
\item Compute the global dimension of $\hqfg$ (compare \cite{CS04}). 
\item Determine those quantum generalized Heisenberg algebras all of whose finite-dimensional representations are completely reducible (compare \cite{CM00}).
\item Study simple weight modules for $\hqfg$. In \cite{LMZ15} the authors set forth the study of simple weight modules over generalized Weyl algebras of non-automorphism type, applying their results to generalized Heisenberg algebras $\hh(f)$. It is natural to extend this study to all quantum generalized Heisenberg algebras $\hqfg$.
\item Determine the primitive ideals of $\hqfg$ (compare \cite{iP09} and \cite{iP11}).
\item Investigate Hochschild (co)homology of $\hqfg$ (compare \cite{CHS18}).
\item In \cite{BW01} the authors embed a certain type of down-up algebra $A$ into a skew group algebra $A\ast G$, where $G$ is a subgroup of the automorphism group of $A$, and then construct a Hopf algebra structure on $A\ast G$. As an example, when the defining parameters of $A$ are such that $A$ is isomorphic to the enveloping algebra of the Heisenberg Lie algebra, then the group $G$ is trivial and the Hopf structure obtained agrees with the usual one on an enveloping algebra. It would be a very interesting project to generalize this construction to (an appropriate subclass of) quantum generalized Heisenberg algebras.
\item In a different direction, in \cite{KM03} the authors classify the down-up algebras which have a Hopf algebra structure and this would be a natural question for $\hqfg$ as well.
\item Going back to the Physics literature on generalized Heisenberg algebras, where these appeared to be defined over rings more general than the polynomial ring $\FF[h]$, it would also be of interest to consider quantum generalized Heisenberg algebras defined over a Laurent polynomial ring $\FF[h^{\pm 1}]$, a power series ring $\FF[[h]]$ or the rational function field $\FF(h)$.
\end{itemize}

\newcommand{\germ}{\mathfrak}
\def\cprime{$'$} \def\cprime{$'$} \def\cprime{$'$}
\bibliographystyle{plain}

\begin{thebibliography}{10}

\bibitem{BCG18}
F.~Bagarello, E.~M.~F. Curado, and J.~P. Gazeau.
\newblock Generalized {H}eisenberg algebra and (non linear) pseudo-bosons.
\newblock {\em J. Phys. A}, 51(15):155201, 16, 2018.

\bibitem{vb91}
V.~V. Bavula.
\newblock Finite-dimensionality of {${\rm Ext}^n$} and {${\rm Tor}_n$} of
  simple modules over a class of algebras.
\newblock {\em Funktsional. Anal. i Prilozhen.}, 25(3):80--82, 1991.

\bibitem{vb92}
V.~V. Bavula.
\newblock Generalized {W}eyl algebras and their representations.
\newblock {\em Algebra i Analiz}, 4(1):75--97, 1992.

\bibitem{BJ01}
V.~V. Bavula and D.~A. Jordan.
\newblock Isomorphism problems and groups of automorphisms for generalized
  {W}eyl algebras.
\newblock {\em Trans. Amer. Math. Soc.}, 353(2):769--794, 2001.

\bibitem{BR98}
G.~Benkart and T.~Roby.
\newblock Down-up algebras.
\newblock {\em J. Algebra}, 209(1):305--344, 1998.

\bibitem{gB99}
Georgia Benkart.
\newblock Down-up algebras and {W}itten's deformations of the universal
  enveloping algebra of {$\germ s\germ l_2$}.
\newblock In {\em Recent progress in algebra ({T}aejon/{S}eoul, 1997)}, volume
  224 of {\em Contemp. Math.}, pages 29--45. Amer. Math. Soc., Providence, RI,
  1999.

\bibitem{BW01}
Georgia Benkart and Sarah Witherspoon.
\newblock A {H}opf structure for down-up algebras.
\newblock {\em Math. Z.}, 238(3):523--553, 2001.

\bibitem{BBH11}
K.~Berrada, M.~El Baz, and Y.~Hassouni.
\newblock Generalized {H}eisenberg algebra coherent states for power-law
  potentials.
\newblock {\em Physics Letters A}, 375(3):298 -- 302, 2011.

\bibitem{CL09}
Paula A. A.~B. Carvalho and Samuel~A. Lopes.
\newblock Automorphisms of generalized down-up algebras.
\newblock {\em Comm. Algebra}, 37(5):1622--1646, 2009.

\bibitem{CM00}
Paula A. A.~B. Carvalho and Ian~M. Musson.
\newblock Down-up algebras and their representation theory.
\newblock {\em J. Algebra}, 228(1):286--310, 2000.

\bibitem{CS04}
Thomas Cassidy and Brad Shelton.
\newblock Basic properties of generalized down-up algebras.
\newblock {\em J. Algebra}, 279(1):402--421, 2004.

\bibitem{CHS18}
Sergio Chouhy, Estanislao Herscovich, and Andrea Solotar.
\newblock Hochschild homology and cohomology of down-up algebras.
\newblock {\em J. Algebra}, 498:102--128, 2018.

\bibitem{CRM01}
E~M~F Curado and M~A Rego-Monteiro.
\newblock Multi-parametric deformed {H}eisenberg algebras: a route to
  complexity.
\newblock {\em Journal of Physics A: Mathematical and General}, 34(15):3253,
  2001.

\bibitem{CRMR13}
E.~M.~F. Curado, M.~A. Rego-Monteiro, and Ligia M. C.~S. Rodrigues.
\newblock Structure of generalized {H}eisenberg algebras and quantum
  decoherence analysis.
\newblock {\em Phys. Rev. A}, 87:052120, May 2013.

\bibitem{GW89}
K.R. Goodearl and R.B. Warfield, Jr.
\newblock {\em An {I}ntroduction to {N}oncommutative {N}oetherian {R}ings},
  volume~16 of {\em London Mathematical Society Student Texts}.
\newblock Cambridge University Press, Cambridge, 1989.

\bibitem{dJ00}
David~A. Jordan.
\newblock Down-up algebras and ambiskew polynomial rings.
\newblock {\em J. Algebra}, 228(1):311--346, 2000.

\bibitem{dJ19}
David~A. Jordan.
\newblock Simple ambiskew polynomial rings {II}: non-bijective endomorphisms.
\newblock In {\em Rings, modules and codes}, volume 727 of {\em Contemp.
  Math.}, pages 177--200. Amer. Math. Soc., Providence, RI, 2019.

\bibitem{JW96}
David~A. Jordan and Imogen~E. Wells.
\newblock Invariants for automorphisms of certain iterated skew polynomial
  rings.
\newblock {\em Proc. Edinburgh Math. Soc. (2)}, 39(3):461--472, 1996.

\bibitem{JW13}
David~A. Jordan and Imogen~E. Wells.
\newblock Simple ambiskew polynomial rings.
\newblock {\em J. Algebra}, 382:46--70, 2013.

\bibitem{KM03}
Ellen~E. Kirkman and Ian~M. Musson.
\newblock Hopf down-up algebras.
\newblock {\em J. Algebra}, 262(1):42--53, 2003.

\bibitem{lLB95}
Lieven Le~Bruyn.
\newblock Conformal {${\rm sl}_2$} enveloping algebras.
\newblock {\em Comm. Algebra}, 23(4):1325--1362, 1995.

\bibitem{sL17}
Samuel~A. Lopes.
\newblock Non-{N}oetherian generalized {H}eisenberg algebras.
\newblock {\em J.\ Algebra Appl.}, 16(2):1750064, 2017.

\bibitem{LR20pr}
Samuel~A. Lopes and Farrokh Razavinia.
\newblock Quantum generalized {H}eisenberg algebras and their isomorphisms.
\newblock In preparation.

\bibitem{LMZ15}
Rencai L{\"u}, Volodymyr Mazorchuk, and Kaiming Zhao.
\newblock Simple weight modules over weak generalized {W}eyl algebras.
\newblock {\em J. Pure Appl. Algebra}, 219(8):3427--3444, 2015.

\bibitem{LZ15}
Rencai L{\"u} and Kaiming Zhao.
\newblock Finite-dimensional simple modules over generalized {H}eisenberg
  algebras.
\newblock {\em Linear Algebra Appl.}, 475:276--291, 2015.

\bibitem{iP09}
Iwan Praton.
\newblock Simple modules and primitive ideals of non-{N}oetherian generalized
  down-up algebras.
\newblock {\em Comm. Algebra}, 37(3):811--839, 2009.

\bibitem{iP11}
Iwan Praton.
\newblock Primitive ideals of {N}oetherian generalized down-up algebras.
\newblock {\em Comm. Algebra}, 39(11):4289--4318, 2011.

\bibitem{sR02}
Sonia Rueda.
\newblock Some algebras similar to the enveloping algebra of {$\rm sl(2)$}.
\newblock {\em Comm. Algebra}, 30(3):1127--1152, 2002.

\bibitem{spS90}
S.~P. Smith.
\newblock A class of algebras similar to the enveloping algebra of {${\rm
  sl}(2)$}.
\newblock {\em Trans. Amer. Math. Soc.}, 322(1):285--314, 1990.

\end{thebibliography}

\end{document}